\documentclass[a4paper,english,oneside]{article}
\usepackage[T1]{fontenc}
\usepackage[utf8]{inputenc} 

\usepackage{geometry}
\usepackage{amsmath}
\usepackage{amsfonts}
\usepackage{amssymb}
\usepackage{amsthm}
\usepackage[english]{babel}
\usepackage{color}
\usepackage{lmodern}
\usepackage{pstricks-add}
\usepackage{hyperref}
\hypersetup{hidelinks}

\usepackage{mathscinet}
\usepackage{enumitem}
\usepackage[cal=boondoxo,bb=ams]{mathalfa}
\usepackage{microtype}

\usepackage{mathtools}
\usepackage{esint}

\theoremstyle{plain}
\newtheorem{theorem}{Theorem}[section]
\newtheorem{lemma}[theorem]{Lemma}
\newtheorem{corollary}[theorem]{Corollary}
\newtheorem{proposition}[theorem]{Proposition}

\theoremstyle{definition}

\theoremstyle{remark}

    \DeclareMathOperator\supp{supp}
    
    \DeclareMathOperator\h{h}
    \DeclareMathOperator\ih{ih}
    \DeclareMathOperator\Fuj{Fuj}
    \DeclareMathOperator\Str{Str}
    \DeclareMathOperator\Huy{Huy}
    \DeclareMathOperator\nHuy{nHuy}
     
\begin{document}

\title{Integral representation formulae for the solution of a wave equation with time-dependent damping and mass in the scale-invariant case}

\author{Alessandro Palmieri \footnote{email: alessandro.palmieri.math@gmail.com}}

\date{\small{ Department of Mathematics, University of Pisa, Largo B. Pontecorvo 5, 56127 Pisa, Italy} \\ [2ex] \normalsize{\today} }
\maketitle

\begin{abstract}
This paper is devoted to derive integral representation formulae for the solution of an inhomogeneous linear wave equation with time-dependent damping and mass terms, that are scale-invariant with respect the so-called hyperbolic scaling. Yagdjian's integral transform approach is employed for this purpose. The main step in our argument consists in determining the kernel functions for the different integral terms, which are related to the source term and to initial data. We will start with the one dimensional case (in space). We point out that we may not apply in a straightforward way Duhamel's principle to deal with the source term since the coefficients of lower order terms make our model not invariant by time translation. On the contrary, we shall begin with the representation formula for the inhomogeneous equation with vanishing data by using a revised Duhamel's principle. Then, we will derive the representation of the solution in the homogeneous case with nontrivial data. After deriving the formula in the one dimensional case, the classical approach by  spherical means is used in order to deal with the odd dimensional case. Finally, using the method of descent, the representation formula in the even dimensional case is proved. 
\end{abstract}

\begin{flushleft}
\end{flushleft}

\begin{flushleft}
\textbf{AMS Classification (2010)} Primary:  35A08, 35C15; Secondary: 33C05, 35L05, 35L15
\end{flushleft}

\section{Introduction}

In the last years, several papers have been devoted to the study of the semilinear wave equations (and weakly coupled systems) with time-dependent damping and mass and power nonlinearity in the scale-invariant case, namely,
\begin{align}\label{semilinear CP}
\begin{cases} u_{tt}-\Delta u +\frac{\mu}{1+t}u_t+\frac{\nu^2}{(1+t)^2}u=|u|^p, &  x\in \mathbb{R}^n, \ t>0,\\
u(0,x)=\varepsilon u_0(x), & x\in \mathbb{R}^n, \\ u_t(0,x)=\varepsilon u_1(x), & x\in \mathbb{R}^n,
\end{cases}
\end{align} where $\mu,\nu^2$ are non negative constants, $p>1$ and $\varepsilon$ is positive constant describing the smallness of Cauchy data (cf. \cite{Abb15,Wak14,DLR15,DabbLuc15,Wak16,LTW17,IS17,TL1709,KatoSak18} for the massless case and \cite{ChenPal18,Pal19} for the weakly coupled system). If we introduce the quantity
\begin{align}
\delta\doteq (\mu-1)^2-4 \nu^2, \label{def delta}
\end{align} then, it is known that the critical exponent for \eqref{semilinear CP} depends on $\delta$. 

On the one hand, for $\delta\geqslant (n+1)^2$ the critical exponent is $p_{\Fuj}\left(n+\frac{\mu-1}{2}-\frac{\sqrt{\delta}}{2}\right)$, where $p_{\Fuj}(n)\doteq 1+\frac{2}{n}$ is the Fujita exponent, see \cite{NPR16,Pal17,PalRei17}. On the other hand, for $\delta$ nonnegative and sufficiently small (depending on the spatial dimension $n$), it has been proved that for any $1<p\leqslant p_{\Str}(n+\mu)$ local in time solutions of \eqref{semilinear CP} blow up in finite times under suitable integral sign assumptions on initial data (see \cite{PalRei18,PT18}). Here $p_{\Str}(n)$ denotes the Strauss exponent, that is the positive root of the quadratic equation $(n-1)p^2-(n+1)p-2=0$. The global in time existence of small data solutions for $p> p_{\Str}(n+\mu)$ has been proved only in the special case $\delta=1$ for radial symmetric solutions and for dimensions $n\geqslant 3$ (cf. \cite{Pal18odd} for the case $n$ odd and \cite{Pal18even} for the case $n$ even). However, in the general case $\delta \neq 1$, the global in time existence of small data solutions for $p> p_{\Str}(n+\mu)$ is still open. Furthermore, in the case $\delta<0$ both the blow-up part and the global (in time) existence part are open, although a partial result is proved for the necessity part in \cite{DabbPal18}.

In the proof of the global existence results for $\delta\geqslant (n+1)^2$ (when the critical exponent is the shift of Fujita exponent), $L^2-L^2$ estimates for the solution of the corresponding homogeneous equation and for its derivatives play a fundamental role. In particular, these estimates are derived by using the explicit representation formula of the fundamental solutions of the corresponding homogeneous problem, which contains in their expression some cylindrical functions due to the scale-invariance of the model. In other terms, it is used an approach based on Fourier integral operators.

In some sense, the fact that the above mentioned shift of Fujita exponent is critical for large values of $\delta$ can be proved by using tools which are suitable for the semilinear classical damped wave equation, such as $L^2-L^2$ decay estimates with additional regularity for initial data for the global existence part or scaling arguments for the blow-up part (namely, the so-called test function method, cf. \cite{MPbook}). Unfortunately, this tools are not suitable when the behavior of the semilinear model in \eqref{semilinear CP} is closer to the semilinear wave equation ($\delta$ nonnegative and ``small'') and we expect to find as critical exponent a shift of Strauss  exponent. Therefore, it might be useful to derive results and tools which are widely employed to deal with the classical wave equation. 

After this preface, we understand why it could be useful to derive an explicit integral representation formula for the solution of the linear wave equation with time-dependent damping and mass term in the scale-invariant case.
More specifically, in this paper we will derive an explicit representation formula for the solution of the linear Cauchy problem 
\begin{align}\label{inhomog CP}
\begin{cases} u_{tt}-\Delta u +\frac{\mu}{1+t}u_t+\frac{\nu^2}{(1+t)^2}u=f(t,x), &  x\in \mathbb{R}^n, \ t>0,\\
u(0,x)=u_0(x), & x\in \mathbb{R}^n, \\ u_t(0,x)=u_1(x), & x\in \mathbb{R}^n,
\end{cases}
\end{align} where $\mu,\nu^2$ are non negative constants. 

In the series of papers \cite{Yag04,Yag07,YagGal08,Yag09,YagGal09,Yag10,Yag13,Yag15,Yag15MN}, many representations formulae for solutions of Cauchy problems for linear hyperbolic PDEs with variable coefficients have been derived. The general scheme is substantially the same: the representation formula is obtained considering the composition of two operators. The external operator is an integral transformation, whose kernel is determined by the time-dependent coefficients and/or by lower order terms, while the internal operator is a solution operator for a family of parameter dependent Cauchy problems (this is somehow a \emph{revised Duhamel's priciple}). In particular, if the considered PDE is a wave equation with time-dependent speed of propagation, then, this solution operator maps a given function into the solution of the Cauchy problem for the classical free wave equation with the given function as first initial data and with vanishing second initial data.

Using \emph{Yagdjian's Integral Transform approach}, we will provide an explicit representation formula for the solution of \eqref{inhomog CP} in all spatial dimensions. More specifically, we begin by studying the one-dimensional case; then, we get the representation formula for odd dimensions via spherical means' method and, finally, by method of descent we find the representation formula for even dimensions.

Let us state the main results of this paper. We start with the case $n=1$. 

\begin{theorem} \label{Thm representation formula 1d case}
Let $n=1$ and let $\mu,\nu^2$ be nonnegative constants. Let us assume $f\in \mathcal{C}^{0,1}_{t,x}([0,\infty)\times \mathbb{R})$ and $u_0\in \mathcal{C}^2(\mathbb{R})$, $u_1\in \mathcal{C}^1(\mathbb{R})$. Then, a representation formula for the solution of \eqref{inhomog CP} is given by
\begin{align}
u(t,x) &= \frac{1}{2}(1+t)^{-\frac{\mu}{2}}\big(u_0(x+t)+u_0(x-t)\big)+\frac{1}{2^{\sqrt{\delta}}}\int_{x-t}^{x+t} u_0(y) K_0(t,x;y;\mu,\nu^2)\, \mathrm{d}y \notag \\ & \quad +\frac{1}{2^{\sqrt{\delta}}}\int_{x-t}^{x+t}\big(u_1(y)+\mu \, u_0(y)\big) K_1(t,x;y;\mu,\nu^2)\, \mathrm{d}y  +\frac{1}{2^{\sqrt{\delta}}} \int_0^t \int_{x-t+b}^{x+t-b} f(b,y) E(t,x;b,y;\mu,\nu^2) \, \mathrm{d}y\, \mathrm{d}b, \label{representation formula 1d case}
\end{align} where the kernel functions are defined as follows
\begin{align}
E(t,x;b,y;\mu,\nu^2) & \doteq (1+t)^{-\frac{\mu}{2}+\frac{1-\sqrt{\delta}}{2}} (1+b)^{\frac{\mu}{2}+\frac{1-\sqrt{\delta}}{2}} \left((t+b+2)^2-(y\!-\!x)^2\right)^{\frac{\sqrt{\delta}-1}{2}}  \mathsf{F}\left(\tfrac{1-\sqrt{\delta}}{2},\tfrac{1-\sqrt{\delta}}{2};1; \tfrac{(t-b)^2-(y-x)^2}{(t+b+2)^2-(y-x)^2} \right), \label{def E(t,x;b,y)}\\
K_0(t,x;y;\mu,\nu^2) & \doteq -\frac{\partial}{\partial b}\, E(t,x;b,y;\mu,\nu^2) \Big|_{b=0}, \label{def K0(t,x;y)}\\
K_1(t,x;y;\mu,\nu^2) & \doteq  E(t,x;0,y;\mu,\nu^2) \label{def K1(t,x;y)}
\end{align} and $\mathsf{F}(\alpha,\beta;\gamma; z)$ denotes Gauss hypergeometric function.
\end{theorem}



Before stating the representation formula in the multidimensional case, let us introduce the following notations: if $f=f(t,x)$ is defined for $t\geqslant 0, x\in\mathbb{R}^n$, then, we denote by $w[f]=w[f](t,x;b)$ the solution to the parameter dependent Cauchy problem for the free wave equation
\begin{align}\label{free wave eq CP parameter b}
\begin{cases} w_{tt}-\Delta w=0, &  x\in \mathbb{R}^n, \ t>0,\\
w(0,x)=f(b,x), & x\in \mathbb{R}^n, \\ w_t(0,x)=0, & x\in \mathbb{R}^n,
\end{cases}
\end{align}
with parameter $b\geqslant 0$. When the function $f$ depends only on the spatial variable, the Cauchy problem depends no longer on the parameter $b$, namely, if $\varphi=\varphi (x)$, then, $w[\varphi]=w[\varphi](t,x)$ denotes the solution to the Cauchy problem for the free wave equation
\begin{align}\label{free wave eq CP}
\begin{cases} w_{tt}-\Delta w=0, &  x\in \mathbb{R}^n, \ t>0,\\
w(0,x)=\varphi(x), & x\in \mathbb{R}^n, \\ w_t(0,x)=0, & x\in \mathbb{R}^n.
\end{cases}
\end{align}

Assuming that the function $f$ (resp. $\varphi$) is sufficiently smooth with respect to the spatial variable, then, the representation formula for $w[f]$ (resp. $w[\varphi]$) is well-known and depends on the parity of $n$ (see for example \cite[Section 2.4]{Evans}). More precisely, when $n\geqslant 3$ is an odd integer it holds
\begin{equation}
\begin{split}
w[f](t,x;b) & =\frac{1}{(n-2)!!} \bigg(\frac{\partial}{\partial t}\bigg) \bigg(\frac{1}{t}\frac{\partial}{\partial t}\bigg)^{\frac{n-3}{2}} \left(t^{n-2} \fint_{\partial B_t(x)} f(b,z) \,\mathrm{d}\sigma_z\right), \\  w[\varphi](t,x) & =\frac{1}{(n-2)!!} \bigg(\frac{\partial}{\partial t}\bigg) \bigg(\frac{1}{t}\frac{\partial}{\partial t}\bigg)^{\frac{n-3}{2}} \left(t^{n-2} \fint_{\partial B_t(x)} \varphi(z) \, \mathrm{d}\sigma_z\right),
\end{split} \label{wave sol operator n odd}
\end{equation} while if $n\geqslant 2$ is an even integer, then,
\begin{equation}
\begin{split}
w[f](t,x;b) & =\frac{1}{n!!} \bigg(\frac{\partial}{\partial t}\bigg) \bigg(\frac{1}{t}\frac{\partial}{\partial t}\bigg)^{\frac{n-2}{2}} \left(t^{n} \fint_{B_t(x)} \frac{f(b,z)}{(t^2-|z-x|^2)^{1/2}} \, \mathrm{d}z \right), \\  w[\varphi](t,x) & =\frac{1}{n!!} \bigg(\frac{\partial}{\partial t}\bigg) \bigg(\frac{1}{t}\frac{\partial}{\partial t}\bigg)^{\frac{n-2}{2}} \left(t^{n} \fint_{B_t(x)} \frac{\varphi(z)}{(t^2-|z-x|^2)^{1/2}} \, \mathrm{d}z \right),
\end{split} \label{wave sol operator n even}
\end{equation} where $\fint_A$ denotes the integral average over $A$ and $j !!$ is the double factorial, which is defined for any $j\in \mathbb{N},j\geq 1$ by
\begin{align*}
j !! \doteq \begin{cases} j(j-2)\cdots 1 & \mbox{if} \ \ j \ \ \mbox{is odd}, \\ j(j-2)\cdots 2 & \mbox{if} \ \ j \ \ \mbox{is even}. \end{cases}
\end{align*}

We may now state the representation formulae in the multidimensional case. We consider separately the case when $n$ is an odd integer and the case  when $n$ is an even integer.

\begin{theorem} \label{Thm representation formula n dimensional odd case} Let $n\geqslant 3$ be an odd integer and let $\mu,\nu^2$ be nonnegative constants. Let us assume $f\in \mathcal{C}^{\frac{n+1}{2}}([0,\infty)\times \mathbb{R}^n)$ and $u_0\in \mathcal{C}^{\frac{n+1}{2}+1}(\mathbb{R}^n)$, $u_1\in \mathcal{C}^{\frac{n+1}{2}}(\mathbb{R}^n)$. Then, a representation formula for the solution of \eqref{inhomog CP} is given by
\begin{align}
u(t,x)& = (1+t)^{-\frac{\mu}{2}} w[u_0](t,x) +\frac{1}{2^{\sqrt{\delta}-1}}\int_0^t w[u_0](s,x) K_0(t,0;s; \mu,\nu^2) \, \mathrm{d}s +\frac{1}{2^{\sqrt{\delta}-1}}\int_0^t w[u_1+\mu\,  u_0](s,x) K_1(t,0;s; \mu,\nu^2) \, \mathrm{d}s \notag \\ & \quad + \frac{1}{2^{\sqrt{\delta}-1}} \int_0^t \int_0^{t-b} w[f](s,x;b) E(t,0;b,s;\mu,\nu^2) \, \mathrm{d}s \, \mathrm{d}b, \label{representation formula n dimensional odd case}
\end{align} where $w[u_0],w[u_1+\mu\, u_0]$ and $w[f]$ are defined by \eqref{wave sol operator n odd}.
\end{theorem}

\begin{theorem} \label{Thm representation formula n dimensional even case} Let $n\geqslant 2$ be an even integer and let $\mu,\nu^2$ be nonnegative constants. Let us assume $f\in \mathcal{C}^{\frac{n}{2}+1}([0,\infty)\times \mathbb{R}^n)$ and $u_0\in \mathcal{C}^{\frac{n}{2}+2}(\mathbb{R}^n)$, $u_1\in \mathcal{C}^{\frac{n}{2}+1}(\mathbb{R}^n)$. Then, a representation formula for the solution of \eqref{inhomog CP} is given by \eqref{representation formula n dimensional odd case}, but with $w[u_0],w[u_1+\mu\, u_0]$ and $w[f]$ defined in this case by \eqref{wave sol operator n even}.
\end{theorem}

The paper is organized as follows: in Section \ref{Section n=1} we prove Theorem \ref{Thm representation formula 1d case} considering first the inhomogeneous problem with vanishing data and, then, we use this case to study the corresponding homogeneous problem with nontrivial data; in Section \ref{Section n odd} we consider the odd dimensional case and we prove Theorem \ref{Thm representation formula n dimensional odd case}; in particular, we use the method of spherical means to associate this case to the one-dimensional one; in Section \ref{Section n even} we consider the even dimensional case and we use the method of descent so that we reduce the problem to the one considered in Section \ref{Section n odd}; finally, in Section \ref{Section final rem} we point out some final remarks to our results and the relations of \eqref{representation formula 1d case} and \eqref{representation formula n dimensional odd case} with the representation formulae for other models with variable coefficients.

\section{One dimensional case} \label{Section n=1}

In this section we will prove Theorem \ref{Thm representation formula 1d case}. Since the Cauchy problem \eqref{inhomog CP} is linear, we may consider separately the case with vanishing initial data and the homogeneous case. In particular, we will show that
\begin{align}
u^{\ih}=u^{\ih}(t,x)=\frac{1}{2^{\sqrt{\delta}}} \int_0^t \int_{x-t+b}^{x+t-b} f(b,y) E(t,x;b,y;\mu,\nu^2) \, \mathrm{d}y\, \mathrm{d}b,  \label{inhomogeneous sol 1d}
\end{align} solves
\begin{align}\label{inhomog CP vanishing data 1d}
\begin{cases} u_{tt}- u_{xx} +\frac{\mu}{1+t}u_t+\frac{\nu^2}{(1+t)^2}u=f(t,x), &  x\in \mathbb{R}, \ t>0,\\
u(0,x)=0, & x\in \mathbb{R}, \\ u_t(0,x)=0, & x\in \mathbb{R},
\end{cases}
\end{align} while
\begin{align}
u^{\h}=u^{\h}(t,x) & =\frac{1}{2}(1+t)^{-\frac{\mu}{2}}\big(u_0(x+t)+u_0(x-t)\big)+\frac{1}{2^{\sqrt{\delta}}}\int_{x-t}^{x+t} u_0(y) K_0(t,x;y;\mu,\nu^2)\, \mathrm{d}y \notag \\ & \quad +\frac{1}{2^{\sqrt{\delta}}}\int_{x-t}^{x+t}\big(u_1(y)+\mu \, u_0(y)\big) K_1(t,x;y;\mu,\nu^2)\, \mathrm{d}y  \label{homogeneous sol 1d}
\end{align} solves
\begin{align}\label{homog CP 1d}
\begin{cases} u_{tt}- u_{xx} +\frac{\mu}{1+t}u_t+\frac{\nu^2}{(1+t)^2}u=0, &  x\in \mathbb{R}, \ t>0,\\
u(0,x)=u_0(x), & x\in \mathbb{R}, \\ u_t(0,x)=u_1(x), & x\in \mathbb{R}.
\end{cases}
\end{align}

The remaining part of the section is organized as follows: in Subsection \ref{Subsection Kernel function E} we prove some fundamental properties of the kernel function $E=E(t,x;b,y)$; then, in Subsection \ref{Subsection inhom 1d CP vanishing data} we prove that $v^{\ih}$ solves \eqref{inhomog CP vanishing data 1d} in the classical sense (punctually); finally, in Subsection \ref{Subsection hom 1d CP} we use the representation formula for the inhomogeneous problem with vanishing data in the 1d case to derive a representation formula for the corresponding  homogeneous case.

\subsection{The kernel function and its properties} \label{Subsection Kernel function E}

In this subsection we investigate some properties of the kernel function $E=E(t,x;b,y)$. Let us begin by proving that $E$ is a solution of the corresponding homogeneous wave equation with scale-invariant damping and mass with respect to the variables $(t,x)$.

For the sake of readability,we introduce the function 
\begin{align} z=z(t,x;b,y) \doteq \frac{(t-b)^2-(y-x)^2}{(t+b+2)^2-(y-x)^2}. \label{def z=z(t,x;b,y)}
\end{align}

\begin{proposition}\label{Prop E is a fund sol} Let $b\in [0,t]$ and $y\in [x-t+b,x+t-b]$. Then,
\begin{align}
\left(\frac{\partial^2}{\partial t^2}-\frac{\partial^2}{\partial x^2}+\frac{\mu}{1+t}\frac{\partial}{\partial t}+\frac{\nu^2}{(1+t)^2}\right)E(t,x;b,y;\mu,\nu^2)=0.\label{formula E is a fund sol}
\end{align}
\end{proposition}

\begin{proof} Let us remark that for $b\in [0,t]$ and $y\in [x-t+b,x+t-b]$ it holds $z=z(t,x;b,y)\in [0,1)$. In particular, we may compute the hypergeometric function in \eqref{def E(t,x;b,y)} without considering the analytic continuation. 
Let us begin by computing the derivatives of $E$ involved in \eqref{formula E is a fund sol}.

\begin{flushleft}
\emph{Representation of $\partial_t^2 E(t,x;b,y;\mu,\nu^2)$ and $\partial_t E(t,x;b,y;\mu,\nu^2)$}
\end{flushleft}

Using the identities
\begin{align*}
 \partial_t \big((t+b+2)^2-(y-x)^2 \big)^{\frac{\sqrt{\delta}-1}{2}} & =  (\sqrt{\delta}-1)(t+b+2) \big((t+b+2)^2-(y-x)^2 \big)^{\frac{\sqrt{\delta}-1}{2}-1}, \\
 \partial_t^2 \big((t+b+2)^2-(y-x)^2 \big)^{\frac{\sqrt{\delta}-1}{2}} & =  (\sqrt{\delta}-1) \big((t+b+2)^2-(y-x)^2 \big)^{\frac{\sqrt{\delta}-1}{2}-1}\\ & \quad +(\sqrt{\delta}-1) (\sqrt{\delta}-1-2) (t+b+2)^2\big((t+b+2)^2-(y-x)^2 \big)^{\frac{\sqrt{\delta}-1}{2}-2}, \\
 \partial_ t \mathsf{F}\left(\tfrac{1-\sqrt{\delta}}{2},\tfrac{1-\sqrt{\delta}}{2};1 ;z\right) &= \mathsf{F}_z\left(\tfrac{1-\sqrt{\delta}}{2},\tfrac{1-\sqrt{\delta}}{2};1 ;z\right) \frac{\partial z}{\partial t}, \\
 \\
 \partial_ t^2 \mathsf{F}\left(\tfrac{1-\sqrt{\delta}}{2},\tfrac{1-\sqrt{\delta}}{2};1 ;z\right) &=  \mathsf{F}_{zz}\left(\tfrac{1-\sqrt{\delta}}{2},\tfrac{1-\sqrt{\delta}}{2};1 ;z\right)  \bigg(\frac{\partial z}{\partial t}\bigg)^2 +  \mathsf{F}_{z}\left(\tfrac{1-\sqrt{\delta}}{2},\tfrac{1-\sqrt{\delta}}{2};1 ;z\right)  \frac{\partial^2 z}{\partial t^2},
\end{align*} we may calculate $\partial_t^2 E(t,x;b,y;\mu,\nu^2)$. Straightforward computations lead to 
\begin{align*}
\frac{\partial^2 E}{\partial t^2}(t,x;b,y;\mu,\nu^2) & = (1+b)^{\frac{\mu}{2}+\frac{1-\sqrt{\delta}}{2}} \partial_t^2 \Big( (1+t)^{-\frac{\mu}{2}+\frac{1-\sqrt{\delta}}{2}}  \left((t+b+2)^2-(y-x)^2\right)^{\frac{\sqrt{\delta}-1}{2}}  \mathsf{F}\Big(\tfrac{1-\sqrt{\delta}}{2},\tfrac{1-\sqrt{\delta}}{2};1; z\Big)\Big) \\
& = (1+b)^{\frac{\mu}{2}+\frac{1-\sqrt{\delta}}{2}} (1+t)^{-\frac{\mu}{2}+\frac{1-\sqrt{\delta}}{2}}  \left((t+b+2)^2-(y-x)^2\right)^{\frac{\sqrt{\delta}-1}{2}} \\
 & \quad \times  \bigg[   \mathsf{F}_{zz}\Big(\tfrac{1-\sqrt{\delta}}{2},\tfrac{1-\sqrt{\delta}}{2};1; z\Big) \, \bigg(\frac{\partial z}{\partial t}\bigg)^2 +   \mathsf{F}_{z}\Big(\tfrac{1-\sqrt{\delta}}{2},\tfrac{1-\sqrt{\delta}}{2};1; z\Big) \, \frac{\partial^2 z}{\partial t^2} \\ & \qquad \quad +  \big(-\tfrac{\mu}{2}+\tfrac{1-\sqrt{\delta}}{2}\big)\big(-\tfrac{\mu}{2}-1+\tfrac{1-\sqrt{\delta}}{2}\big) (1+t)^{-2}  \mathsf{F}\Big(\tfrac{1-\sqrt{\delta}}{2},\tfrac{1-\sqrt{\delta}}{2};1; z\Big) \\   
 & \qquad \quad + (\sqrt{\delta}-1)   \left((t+b+2)^2-(y-x)^2\right)^{-1}  \mathsf{F}\Big(\tfrac{1-\sqrt{\delta}}{2},\tfrac{1-\sqrt{\delta}}{2};1; z\Big)\\   
 & \qquad \quad + (\sqrt{\delta}-1) (\sqrt{\delta}-1-2)   \left((t+b+2)^2-(y-x)^2\right)^{-2} (t+b+2)^2\,   \mathsf{F}\Big(\tfrac{\sqrt{\delta}-1}{2},\tfrac{\sqrt{\delta}-1}{2};1; z\Big) \\   
  & \qquad \quad + 2 (-\tfrac{\mu}{2}+\tfrac{1-\sqrt{\delta}}{2})(\sqrt{\delta}-1) (1+t)^{-1}  \left((t+b+2)^2-(y-x)^2\right)^{-1} (t+b+2)\,   \mathsf{F}\Big(\tfrac{1-\sqrt{\delta}}{2},\tfrac{1-\sqrt{\delta}}{2};1; z\Big) \\   
  & \qquad \quad + 2 (-\tfrac{\mu}{2}+\tfrac{1-\sqrt{\delta}}{2})(1+t)^{-1}  \mathsf{F}_z\Big(\tfrac{1-\sqrt{\delta}}{2},\tfrac{1-\sqrt{\delta}}{2};1; z\Big) \, \frac{\partial z}{\partial t}  \\   
   & \qquad \quad +2 (\sqrt{\delta}-1) \left((t+b+2)^2-(y-x)^2\right)^{-1} (t+b+2) \, \mathsf{F}_z\Big(\tfrac{1-\sqrt{\delta}}{2},\tfrac{1-\sqrt{\delta}}{2};1; z\Big) \, \frac{\partial z}{\partial t} \bigg]
\end{align*} and, similarly, to 
\begin{align}
\frac{\partial E}{\partial t} (t,x;b,y;\mu,\nu^2) & = (1+b)^{\frac{\mu}{2}+\frac{1-\sqrt{\delta}}{2}} \partial_t \Big( (1+t)^{-\frac{\mu}{2}+\frac{1-\sqrt{\delta}}{2}}  \left((t+b+2)^2-(y-x)^2\right)^{\frac{\sqrt{\delta}-1}{2}}  \mathsf{F}\Big(\tfrac{1-\sqrt{\delta}}{2},\tfrac{1-\sqrt{\delta}}{2};1; z\Big)\Big) \notag\\
& = (1+b)^{\frac{\mu}{2}+\frac{1-\sqrt{\delta}}{2}} (1+t)^{-\frac{\mu}{2}+\frac{1-\sqrt{\delta}}{2}}  \left((t+b+2)^2-(y-x)^2\right)^{\frac{\sqrt{\delta}-1}{2}} \notag\\
 & \quad \times  \bigg[ \mathsf{F}_{z}\Big(\tfrac{1-\sqrt{\delta}}{2},\tfrac{1-\sqrt{\delta}}{2};1; z\Big) \, \frac{\partial z}{\partial t} + \big(-\tfrac{\mu}{2}+\tfrac{1-\sqrt{\delta}}{2}\big)(1+t)^{-1}  \mathsf{F}\Big(\tfrac{1-\sqrt{\delta}}{2},\tfrac{1-\sqrt{\delta}}{2};1; z\Big)\notag \\   
 & \qquad \quad + (\sqrt{\delta}-1)   \left((t+b+2)^2-(y-x)^2\right)^{-1} (t+b+2) \,  \mathsf{F}\Big(\tfrac{1-\sqrt{\delta}}{2},\tfrac{1-\sqrt{\delta}}{2};1; z\Big)  \bigg]. \label{dE/dt}
\end{align}

\begin{flushleft}
\emph{Representation of $\partial_x^2 E(t,x;b,y;\mu,\nu^2)$ }
\end{flushleft} 

In order to calculate the partial derivative $\partial_x^2 E(t,x;b,y;\mu,\nu^2)$, we will employ the following relations
\begin{align*}
 \partial_x \big((t+b+2)^2-(y-x)^2 \big)^{\frac{\sqrt{\delta}-1}{2}} & =  (\sqrt{\delta}-1)(y-x) \big((t+b+2)^2-(y-x)^2 \big)^{\frac{\sqrt{\delta}-1}{2}-1}, \\
 \partial_x^2 \big((t+b+2)^2-(y-x)^2 \big)^{\frac{\sqrt{\delta}-1}{2}} & = (\sqrt{\delta}-1) (\sqrt{\delta}-1-2) (y-x)^2\big((t+b+2)^2-(y-x)^2 \big)^{\frac{\sqrt{\delta}-1}{2}-2} \\ & \quad -(\sqrt{\delta}-1) \big((t+b+2)^2-(y-x)^2 \big)^{\frac{\sqrt{\delta}-1}{2}-1}, \\
 \partial_ x \mathsf{F}\left(\tfrac{1-\sqrt{\delta}}{2},\tfrac{1-\sqrt{\delta}}{2};1 ;z\right) &= \mathsf{F}_z\left(\tfrac{1-\sqrt{\delta}}{2},\tfrac{1-\sqrt{\delta}}{2};1 ;z\right) \frac{\partial z}{\partial x}, \\
 \\
 \partial_ x^2 \mathsf{F}\left(\tfrac{1-\sqrt{\delta}}{2},\tfrac{1-\sqrt{\delta}}{2};1 ;z\right) &=  \mathsf{F}_{zz}\left(\tfrac{1-\sqrt{\delta}}{2},\tfrac{1-\sqrt{\delta}}{2};1 ;z\right)  \bigg(\frac{\partial z}{\partial x}\bigg)^2 +  \mathsf{F}_{z}\left(\tfrac{1-\sqrt{\delta}}{2},\tfrac{1-\sqrt{\delta}}{2};1 ;z\right) \frac{\partial^2 z}{\partial x^2}.
\end{align*} Then,
\begin{align*}
\frac{\partial^2 E}{\partial x^2}(t,x;b,y;\mu,\nu^2) & = (1+b)^{\frac{\mu}{2}+\frac{1-\sqrt{\delta}}{2}} (1+t)^{-\frac{\mu}{2}+\frac{1-\sqrt{\delta}}{2}} \partial_x ^2 \Big( \left((t+b+2)^2-(y-x)^2\right)^{\frac{\sqrt{\delta}-1}{2}}  \mathsf{F}\Big(\tfrac{1-\sqrt{\delta}}{2},\tfrac{1-\sqrt{\delta}}{2};1; z\Big)\Big) \\
& = (1+b)^{\frac{\mu}{2}+\frac{1-\sqrt{\delta}}{2}} (1+t)^{-\frac{\mu}{2}+\frac{1-\sqrt{\delta}}{2}}  \left((t+b+2)^2-(y-x)^2\right)^{\frac{\sqrt{\delta}-1}{2}} \\
 & \quad \times  \bigg[   \mathsf{F}_{zz}\Big(\tfrac{1-\sqrt{\delta}}{2},\tfrac{1-\sqrt{\delta}}{2};1; z\Big) \, \bigg(\frac{\partial z}{\partial x}\bigg)^2 +   \mathsf{F}_{z}\Big(\tfrac{1-\sqrt{\delta}}{2},\tfrac{1-\sqrt{\delta}}{2};1; z\Big) \, \frac{\partial^2 z}{\partial x^2} \\   
 & \qquad \quad + (\sqrt{\delta}-1) (\sqrt{\delta}-1-2)   \left((t+b+2)^2-(y-x)^2\right)^{-2} (y-x)^2\,   \mathsf{F}\Big(\tfrac{1-\sqrt{\delta}}{2},\tfrac{1-\sqrt{\delta}}{2};1; z\Big) \\   
 & \qquad \quad - (\sqrt{\delta}-1)   \left((t+b+2)^2-(y-x)^2\right)^{-1}  \mathsf{F}\Big(\tfrac{1-\sqrt{\delta}}{2},\tfrac{1-\sqrt{\delta}}{2};1; z\Big) \\
   & \qquad \quad +2 (\sqrt{\delta}-1) \left((t+b+2)^2-(y-x)^2\right)^{-1} (y-x) \, \mathsf{F}_z\Big(\tfrac{1-\sqrt{\delta}}{2},\tfrac{1-\sqrt{\delta}}{2};1; z\Big) \, \dfrac{\partial z}{\partial x} \bigg].
\end{align*}

Combining now the expressions for the derivatives of $E$, we can now prove \eqref{formula E is a fund sol}. Collecting the similar terms, we get 
\begin{align}
& \frac{\partial^2 E}{\partial t^2}  (t,x;b,y;\mu,\nu^2)- \frac{\partial^2 E}{\partial x^2} (t,x;b,y;\mu,\nu^2) +\frac{\mu}{1+t} \frac{\partial E}{\partial t}(t,x;b,y;\mu,\nu^2) +\frac{\nu^2}{(1+t)^2}  E(t,x;b,y;\mu,\nu^2)\notag \\
& \quad = (1+b)^{\frac{\mu}{2}+\frac{1-\sqrt{\delta}}{2}} (1+t)^{-\frac{\mu}{2}+\frac{1-\sqrt{\delta}}{2}}  \left((t+b+2)^2-(y-x)^2\right)^{\frac{\sqrt{\delta}-1}{2}} \notag \\
 & \quad \quad \times  \bigg\{ \bigg[\bigg(\frac{\partial z}{\partial t}\bigg)^2-\bigg(\frac{\partial z}{\partial x}\bigg)^2\, \bigg] \mathsf{F}_{zz}\left(\tfrac{1-\sqrt{\delta}}{2},\tfrac{1-\sqrt{\delta}}{2};1 ;z\right) \notag \\ 
 & \qquad \qquad +\bigg[\frac{\partial^2 z}{\partial t^2}-\frac{\partial^2 z}{\partial x^2}+ \big(2\big(-\tfrac{\mu}{2}+\tfrac{1-\sqrt{\delta}}{2}\big)+\mu\big)(1+t) ^{-1}\frac{\partial z}{\partial t} \notag\\ 
 & \quad \qquad \qquad \ \ +2(\sqrt{\delta}-1) \left((t+b+2)^2-(y-x)^2\right)^{-1}\bigg((t+b+2)\,\frac{\partial z}{\partial t}-(y-x) \, \frac{\partial z}{\partial x}\bigg)\bigg] \mathsf{F}_{z}\left(\tfrac{1-\sqrt{\delta}}{2},\tfrac{1-\sqrt{\delta}}{2};1 ;z\right) \notag
 \\
 & \qquad \qquad +\bigg[\underbrace{\Big(\big(-\tfrac{\mu}{2}+\tfrac{1-\sqrt{\delta}}{2}\big)\big(-\tfrac{\mu}{2}-1+\tfrac{1-\sqrt{\delta}}{2}\big)+\mu \big(-\tfrac{\mu}{2}+\tfrac{1-\sqrt{\delta}}{2}\big)+\nu^2\Big)}_{=0} (1+t) ^{-2} \notag \\ 
 & \quad \qquad \qquad \ \ +\underbrace{\Big(2\big(-\tfrac{\mu}{2}+\tfrac{1-\sqrt{\delta}}{2}\big)(\sqrt{\delta}-1)+\mu (\sqrt{\delta}-1)\Big)}_{-(\sqrt{\delta}-1)^2}(1+t)^{-1}\left((t+b+2)^2-(y-x)^2\right)^{-1} (t+b+2) \notag
 \\ 
 & \quad \qquad \qquad \ \ +(\sqrt{\delta}-1)^2 \left((t+b+2)^2-(y-x)^2\right)^{-1} \bigg] \mathsf{F}\left(\tfrac{1-\sqrt{\delta}}{2},\tfrac{1-\sqrt{\delta}}{2};1 ;z\right)  \bigg\} \notag \\
 & \quad = (1+b)^{\frac{\mu}{2}+\frac{1-\sqrt{\delta}}{2}} (1+t)^{-\frac{\mu}{2}+\frac{1-\sqrt{\delta}}{2}}  \left((t+b+2)^2-(y-x)^2\right)^{\frac{\sqrt{\delta}-1}{2}} \notag \\
 & \quad \quad \times  \bigg\{ \bigg[\bigg(\frac{\partial z}{\partial t}\bigg)^2-\bigg(\frac{\partial z}{\partial x}\bigg)^2\, \bigg] \mathsf{F}_{zz}\left(\tfrac{1-\sqrt{\delta}}{2},\tfrac{1-\sqrt{\delta}}{2};1 ;z\right) \notag  \\ 
 & \qquad \qquad +\bigg[\frac{\partial^2 z}{\partial t^2}-\frac{\partial^2 z}{\partial x^2}+ (1-\sqrt{\delta})(1+t) ^{-1}\frac{\partial z}{\partial t}\notag\\ 
 & \qquad \qquad \quad \ \ +2(\sqrt{\delta}-1) \left((t+b+2)^2-(y-x)^2\right)^{-1}\bigg((t+b+2)\, \frac{\partial z}{\partial t}-(y-x)\, \frac{\partial z}{\partial x}\bigg)\bigg] \mathsf{F}_{z}\left(\tfrac{1-\sqrt{\delta}}{2},\tfrac{1-\sqrt{\delta}}{2};1 ;z\right) \notag
 \\
 & \qquad \qquad +\bigg[(\sqrt{\delta}-1)^2 \left((t+b+2)^2-(y-x)^2\right)^{-1}\notag \\ 
 & \quad \qquad \qquad \ \ -(\sqrt{\delta}-1)^2(1+t)^{-1}\left((t+b+2)^2-(y-x)^2\right)^{-1} (t+b+2)  \bigg] \mathsf{F}\left(\tfrac{1-\sqrt{\delta}}{2},\tfrac{1-\sqrt{\delta}}{2};1 ;z\right)  \bigg\}. \label{intermediate formula E is a fond sol}
\end{align} Next, we will use that $\mathsf{F}\left(\tfrac{1-\sqrt{\delta}}{2},\tfrac{1-\sqrt{\delta}}{2};1 ;z\right)$ solves the differential equation
\begin{align}\label{ODE Hypergeo funct}
z(1-z)\mathsf{F}_{zz}\left(\tfrac{1-\sqrt{\delta}}{2},\tfrac{1-\sqrt{\delta}}{2};1 ;z\right)+\left[1-\left(2-\sqrt{\delta}\right)z\right] \mathsf{F}_z\left(\tfrac{1-\sqrt{\delta}}{2},\tfrac{1-\sqrt{\delta}}{2};1 ;z\right)-\left(\tfrac{1-\sqrt{\delta}}{2}\right)^2 \mathsf{F}\left(\tfrac{1-\sqrt{\delta}}{2},\tfrac{1-\sqrt{\delta}}{2};1 ;z\right)=0.
\end{align} For this purpose, we need first to rewrite the terms in the right hand side of the chain of equalities \eqref{intermediate formula E is a fond sol} that multiply $\mathsf{F}_{zz}\left(\tfrac{1-\sqrt{\delta}}{2},\tfrac{1-\sqrt{\delta}}{2};1 ;z\right)$, $\mathsf{F}_z\left(\tfrac{1-\sqrt{\delta}}{2},\tfrac{1-\sqrt{\delta}}{2};1 ;z\right)$ and $\mathsf{F}\left(\tfrac{1-\sqrt{\delta}}{2},\tfrac{1-\sqrt{\delta}}{2};1 ;z\right)$, respectively. Let us begin with the terms containing derivatives of the function $z=z(t,x;b,y)$. By elementary computations we get
\begin{align*}
\frac{\partial z}{\partial t} (t,x;b,y) &= \frac{4(1+b)\big[(t-b)(t+b+2)+(y-x)^2\big]}{\big[(t+b+2)^2-(y-x)^2\big]^2}  , \\
\frac{\partial^2 z}{\partial t^2}(t,x;b,y) &= \frac{8(1+b)\big[(1+t)\big((t+b+2)^2-(y-x)^2\big)-2\big((t-b)(t+b+2)+(y-x)^2\big)(t+b+2)\big]}{\big[(t+b+2)^2-(y-x)^2\big]^3}   ,\\
\frac{\partial z}{\partial x}(t,x;b,y) &=\frac{8(1+b)(1+t)(y-x)}{\big[(t+b+2)^2-(y-x)^2\big]^2}    , \\
\frac{\partial^2 z}{\partial x^2}(t,x;b,y) &=  -\frac{8(1+b)(1+t)\big[(t+b+2)^2+3(y-x)^2\big]}{\big[(t+b+2)^2-(y-x)^2\big]^3} .
\end{align*}
Let us rewrite the factor that multiplies $\mathsf{F}_{zz}\left(\tfrac{1-\sqrt{\delta}}{2},\tfrac{1-\sqrt{\delta}}{2};1 ;z\right)$ in \eqref{intermediate formula E is a fond sol}. Using the identity 
\begin{align*}
\big((A+B)^2-C^2\big)\big((A-B)^2-C^2\big) & = \big((A+C)^2-B^2\big)\big((A-C)^2-B^2\big)= \big((B+C)^2-A^2\big)\big((B+C)^2-A^2\big)\\  &  = A^4+B^4+C^4-2A^2B^2-2A^2C^2-2B^2C^2  \qquad \mbox{for any} \ \ A,B,C\in \mathbb{R},
\end{align*} it follows
\begin{align}
\bigg(\frac{\partial z}{\partial t}\bigg)^2 -\bigg(\frac{\partial z}{\partial x}\bigg)^2  & = \frac{16(1+b)^2}{\big[(t+b+2)^2-(y-x)^2\big]^4}\Big\{ \big[(t-b)(t+b+2)+(y-x)^2\big]^2- 4(1+t)^2(y-x)^2\Big\} \notag \\
& = \frac{16(1+b)^2}{\big[(t+b+2)^2-(y-x)^2\big]^4}\Big\{ \big[(1+t)^2-(1+b)^2+(y-x)^2\big]^2- 4(1+t)^2(y-x)^2\Big\} \notag \\
& = \frac{16(1+b)^2}{\big[(t+b+2)^2-(y-x)^2\big]^4}\Big\{ \big[((1+t)-(y-x))^2-(1+b)^2\big]\big[((1+t)+(y-x))^2-(1+b)^2\big]\Big\} \notag\\
& = \frac{16(1+b)^2}{\big[(t+b+2)^2-(y-x)^2\big]^4}\Big\{ \big[((1+t)-(1+b))^2-(y-x)^2\big]\big[((1+t)+(1+b))^2-(y-x)^2\big]\Big\} \notag\\
& = \frac{16(1+b)^2\big[(t-b)^2-(y-x)^2\big]}{\big[(t+b+2)^2-(y-x)^2\big]^3}. \label{dz/dt^2 dz/dx^2 formula}
\end{align} Since
\begin{align}
z(1-z) & = \frac{\big[(t-b)^2-(y-x)^2\big]\big[(t+b+2)^2-(t-b)^2\big]}{\big[(t+b+2)^2-(y-x)^2\big]^2}= \frac{4(1+t)(1+b)\big[(t-b)^2-(y-x)^2\big]}{\big[(t+b+2)^2-(y-x)^2\big]^2}, \label{z(1-z) formula}
\end{align} combining \eqref{dz/dt^2 dz/dx^2 formula} and \eqref{z(1-z) formula}, we find
\begin{align}
\bigg(\frac{\partial z}{\partial t}\bigg)^2 -\bigg(\frac{\partial z}{\partial x}\bigg)^2 = \frac{4(1+b)\,  z(1-z)}{(1+t)\big[(t+b+2)^2-(y-x)^2\big]}. \label{coefficient Fzz}
\end{align} We rewrite now the factor multiplying $\mathsf{F}_{z}\left(\tfrac{1-\sqrt{\delta}}{2},\tfrac{1-\sqrt{\delta}}{2};1 ;z\right)$ in the right hand side of \eqref{intermediate formula E is a fond sol}. We remark that
\begin{align*}
\frac{\partial^2 z}{\partial t^2}-\frac{\partial^2 z}{\partial x^2} 
 & = \frac{16(1+b)\big[(1+t)\big((t+b+2)^2+(y-x)^2\big)-\big((t-b)(t+b+2)+(y-x)^2\big)(t+b+2)\big]}{\big[(t+b+2)^2-(y-x)^2\big]^3} \\
 & = \frac{16(1+b)\big[(t+b+2)^2\big((1+t)-(t-b)\big)+(y-x)^2\big((1+t)-(t+b+2)\big)\big]}{\big[(t+b+2)^2-(y-x)^2\big]^3}  = \frac{16(1+b)^2}{\big[(t+b+2)^2-(y-x)^2\big]^2}
\end{align*} and 
\begin{align*}
(t+b+2)\, \frac{\partial z}{\partial t}-(y-x)\, \frac{\partial z}{\partial x} &= (t+b+2)\frac{4(1+b)\big[(t-b)(t+b+2)+(y-x)^2\big]}{\big[(t+b+2)^2-(y-x)^2\big]^2}-(y-x)\frac{8(1+b)(1+t)(y-x)}{\big[(t+b+2)^2-(y-x)^2\big]^2} \\
&= \frac{4(1+b)\big[(t-b)(t+b+2)^2+(y-x)^2\big((t+b+2)-2(1+t)\big]}{\big[(t+b+2)^2-(y-x)^2\big]^2} \\
& =\frac{4(1+b)(t-b)}{\big[(t+b+2)^2-(y-x)^2\big]}.
\end{align*} Therefore,
\begin{align}
& \frac{\partial^2 z}{\partial t^2}-\frac{\partial^2 z}{\partial x^2}+(1+t)^{-1}\frac{\partial z}{\partial t}-2\big((t+b+2)^2-(y-x)^2\big)^{-1}\bigg((t+b+2)\, \frac{\partial z}{\partial t}-(y-x)\, \frac{\partial z}{\partial x}\bigg) \notag\\
& \qquad = \frac{16(1+b)^2}{\big[(t+b+2)^2-(y-x)^2\big]^2} +\frac{4(1+b)\big[(t-b)(t+b+2)+(y-x)^2\big]}{(1+t)\big[(t+b+2)^2-(y-x)^2\big]^2} -\frac{8(1+b)(t-b)}{\big[(t+b+2)^2-(y-x)^2\big]^2} \notag\\
& \qquad = \frac{4(1+b)}{(1+t)\big[(t+b+2)^2-(y-x)^2\big]^2} \big\{ 4(1+b)(1+t)+(t-b)(t+b+2)+(y-x)^2-2(t-b)(1+t)\big\} \notag\\
& \qquad = \frac{4(1+b)}{(1+t)\big[(t+b+2)^2-(y-x)^2\big]^2} \big\{ 4(1+b)(1+t)+(t+b+2)^2-2(1+b)(t+b+2)+(y-x)^2-2(t-b)(1+t)\big\}\notag \\
& \qquad = \frac{4(1+b)}{(1+t)\big[(t+b+2)^2-(y-x)^2\big]^2} \big\{ 2(1+b)\big[2(1+t)-(t+b+2)\big]+(t+b+2)^2+(y-x)^2-2(t-b)(1+t)\big\}\notag \\
& \qquad = \frac{4(1+b)}{(1+t)\big[(t+b+2)^2-(y-x)^2\big]^2} \big\{ 2(1+b)(t-b)+(t+b+2)^2+(y-x)^2-2(t-b)(1+t)\big\} \notag\\
& \qquad = \frac{4(1+b)}{(1+t)\big[(t+b+2)^2-(y-x)^2\big]^2} \big\{ -2(t-b)^2+(t+b+2)^2+(y-x)^2\big\} \label{coefficient Fz intermediate 1}
\end{align} and
\begin{align}
&  -(1+t)^{-1}\frac{\partial z}{\partial t} +2\big((t+b+2)^2-(y-x)^2\big)^{-1}\bigg((t+b+2)\, \frac{\partial z}{\partial t}-(y-x)\, \frac{\partial z}{\partial x}\bigg) \notag\\
& \qquad = -\frac{4(1+b)\big[(t-b)(t+b+2)+(y-x)^2\big]}{(1+t)\big[(t+b+2)^2-(y-x)^2\big]^2} +\frac{8(1+b)(t-b)}{\big[(t+b+2)^2-(y-x)^2\big]^2}\notag \\
& \qquad = \frac{4(1+b)}{(1+t)\big[(t+b+2)^2-(y-x)^2\big]^2} \big\{ -(t-b)(t+b+2)-(y-x)^2+2(t-b)(1+t)\big\} \notag\\
& \qquad = \frac{4(1+b)}{(1+t)\big[(t+b+2)^2-(y-x)^2\big]^2} \big\{ (t-b)^2-(y-x)^2\big\} =  \frac{4(1+b)\, z}{(1+t)\big[(t+b+2)^2-(y-x)^2\big]}. \label{coefficient Fz intermediate 2}
\end{align} Moreover,
\begin{align}
1-2z &= \frac{(t+b+2)^2-2(t-b)^2+(y-x)^2}{(t+b+2)^2-(y-x)^2}
 \label{1-2z formula}.
\end{align} Also, combining \eqref{coefficient Fz intermediate 1}, \eqref{coefficient Fz intermediate 2} and \eqref{1-2z formula}, we have
\begin{align}
& \frac{\partial^2 z}{\partial t^2}-\frac{\partial^2 z}{\partial x^2}+(1-\sqrt{\delta})(1+t)^{-1}\frac{\partial z}{\partial t}+2(\sqrt{\delta}-1)\big((t+b+2)^2-(y-x)^2\big)^{-1}\bigg((t+b+2)\, \frac{\partial z}{\partial t}-(y-x)\, \frac{\partial z}{\partial x}\bigg) \notag\\
& \qquad = \frac{4(1+b)\, (1-2z)}{(1+t)\big[(t+b+2)^2-(y-x)^2\big]} + \frac{4\sqrt{\delta } (1+b)\, z}{(1+t)\big[(t+b+2)^2-(y-x)^2\big]}\notag \\
& \qquad = \frac{4(1+b)\, \big[1-2z+\sqrt{\delta} z\big]}{(1+t)\big[(t+b+2)^2-(y-x)^2\big]}  = \frac{4(1+b)\, \big[1-\big(2-\sqrt{\delta}\big) z\big] }{(1+t)\big[(t+b+2)^2-(y-x)^2\big]},
\label{coefficient Fz}
\end{align} which is the factor that multiplies the term $\mathsf{F}_{z}\left(\tfrac{1-\sqrt{\delta}}{2},\tfrac{1-\sqrt{\delta}}{2};1 ;z\right)$ in \eqref{intermediate formula E is a fond sol}.
Finally, we determine the factor multiplying $\mathsf{F}\left(\tfrac{1-\sqrt{\delta}}{2},\tfrac{1-\sqrt{\delta}}{2};1 ;z\right)$ in \eqref{intermediate formula E is a fond sol}, namely
\begin{align}
(\sqrt{\delta}-1)^2 \left((t+b+2)^2-(y-x)^2\right)^{-1} \bigg(1-\frac{t+b+2}{1+t}\bigg) &= - (\sqrt{\delta}-1)^2 \frac{(1+b)}{(1+t)\big[(t+b+2)^2-(y-x)^2\big]} \notag \\
&= - \frac{4(1+b)}{(1+t)\big[(t+b+2)^2-(y-x)^2\big]} \big(\tfrac{\sqrt{\delta}-1}{2}\big)^2. \label{coefficient F}
\end{align} Summarizing, if we plug \eqref{coefficient Fzz}, \eqref{coefficient Fz} and \eqref{coefficient F} in \eqref{intermediate formula E is a fond sol}, we arrive at 
\begin{align*}
& \frac{\partial^2 E}{\partial t^2}  (t,x;b,y;\mu,\nu^2)- \frac{\partial^2 E}{\partial x^2} (t,x;b,y;\mu,\nu^2) +\frac{\mu}{1+t} \frac{\partial E}{\partial t}(t,x;b,y;\mu,\nu^2) +\frac{\nu^2}{(1+t)^2}  E(t,x;b,y;\mu,\nu^2) \\
& \qquad = \frac{4(1+b)}{(1+t)\big[(t+b+2)^2-(y-x)^2\big]} \bigg\{z(1-z) \mathsf{F}_{zz}\left(\tfrac{1-\sqrt{\delta}}{2},\tfrac{1-\sqrt{\delta}}{2};1 ;z\right)+\big[1-\big(2-\sqrt{\delta}\big) z\big]\mathsf{F}_{z}\left(\tfrac{1-\sqrt{\delta}}{2},\tfrac{1-\sqrt{\delta}}{2};1 ;z\right) \\ & \qquad \qquad \qquad \qquad \qquad \qquad \qquad \qquad \qquad - \big(\tfrac{\sqrt{\delta}-1}{2}\big)^2 \mathsf{F}\left(\tfrac{1-\sqrt{\delta}}{2},\tfrac{1-\sqrt{\delta}}{2};1 ;z\right) \bigg\}=0,
\end{align*} where in the last step we used \eqref{ODE Hypergeo funct}. This completes the proof \eqref{formula E is a fund sol}.
\end{proof}

So far, we proved that the kernel function $E$ is a solution of the homogeneous wave equation with scale-invariant damping and mass with respect to $(t,x)$. As consequence, we prove now that $E$ is a solution of the adjoint equation of the homogeneous wave equation with scale-invariant damping and mass with respect to $(b,y)$.

\begin{corollary}\label{Cor E is a fund sol adj eq} Let $b\in [0,t]$ and $y\in [x-t+b,x+t-b]$. Then,
\begin{align}
\left(\frac{\partial^2}{\partial b^2}-\frac{\partial^2}{\partial y^2}-\frac{\mu}{1+b}\frac{\partial}{\partial b}+\frac{\mu+\nu^2}{(1+t)^2}\right)E(t,x;b,y;\mu,\nu^2)=0.\label{formula E is a fund sol adj eq}
\end{align}
\end{corollary}

\begin{proof} Let us begin by remarking the identity
\begin{align*}
E(t,x;b,y;\mu,\nu^2) = (1+b)^{\mu}(1+t)^{-\mu} E(b,y;t,x;\mu,\nu^2).
\end{align*} Hence, Proposition \ref{Prop E is a fund sol} implies
\begin{align*}
\frac{\partial^2 E}{\partial y^2} (t,x;b,y;\mu,\nu^2) & =(1+b)^{\mu}(1+t)^{-\mu} \frac{\partial^2 E}{\partial y^2} (b,y;t,x;\mu,\nu^2) \\
& =(1+b)^{\mu}(1+t)^{-\mu} \bigg[\frac{\partial^2 }{\partial b^2}  +\frac{\mu}{1+b}\frac{\partial }{\partial b} +\frac{\nu^2}{(1+b)^2} \bigg] E(b,y;t,x;\mu,\nu^2) \\
& =(1+b)^{\mu}(1+t)^{-\mu} \bigg[\frac{\partial^2 }{\partial b^2}  +\frac{\mu}{1+b}\frac{\partial }{\partial b} +\frac{\nu^2}{(1+b)^2} \bigg] \Big((1+b)^{-\mu}(1+t)^{\mu}E(t,x;b,y;\mu,\nu^2) \Big) \\
& =(1+b)^{\mu}\bigg[(1+b)^{-\mu}\frac{\partial^2 }{\partial b^2}-2\mu (1+b)^{-\mu-1} \frac{\partial}{\partial b}+ \mu(\mu+1)(1+b)^{-\mu-2} \\ & \qquad \qquad \quad \  +\mu (1+b)^{-\mu-1}\frac{\partial }{\partial b} -\mu^2 (1+b)^{-\mu-2}+\nu^2(1+b)^{-\mu-2} \bigg] E(t,x;b,y;\mu,\nu^2) \\
& =\bigg[\frac{\partial^2 }{\partial b^2}-\frac{\mu}{1+b}\frac{\partial }{\partial b}+ \frac{\mu+\nu^2}{(1+b)^2} \bigg] E(t,x;b,y;\mu,\nu^2) ,
\end{align*} which is exactly \eqref{formula E is a fund sol adj eq}.
\end{proof}

\begin{lemma} \label{Lemma dt E-dx E + ... =0} Let $b\in [0,t]$. Then,
\begin{align*}
& \bigg[\frac{\partial E}{\partial t}(t,x;b,y;\mu,\nu^2)- \frac{\partial E}{\partial x}(t,x;b,y;\mu,\nu^2)\bigg]_{y=x+t-b}+2^{\sqrt{\delta}-2}\mu (1+t)^{-\frac{\mu}{2}-1}(1+b)^{\frac{\mu}{2}}=0, \\
& \bigg[\frac{\partial E}{\partial t}(t,x;b,y;\mu,\nu^2)+\frac{\partial E}{\partial x}(t,x;b,y;\mu,\nu^2)\bigg]_{y=x-t+b}+2^{\sqrt{\delta}-2}\mu (1+t)^{-\frac{\mu}{2}-1}(1+b)^{\frac{\mu}{2}}=0.
\end{align*}
\end{lemma}

\begin{proof} Using the identity $\mathsf{F}_z(\alpha,\beta;\gamma;z)=\frac{\alpha \beta}{\gamma}\mathsf{F}(\alpha+1,\beta+1;\gamma+1;z)$, by \eqref{dE/dt} it follows
\begin{align*}
\frac{\partial E}{\partial t}(t,x;b,y;\mu,\nu^2) 
& = (1+b)^{\frac{\mu}{2}+\frac{1-\sqrt{\delta}}{2}} (1+t)^{-\frac{\mu}{2}+\frac{1-\sqrt{\delta}}{2}}  \left((t+b+2)^2-(y-x)^2\right)^{\frac{\sqrt{\delta}-1}{2}} \\
 & \quad \times  \bigg[ (1-\sqrt{\delta})^2(1+b)\big((t-b)(t+b+2)+(y-x)^2\big)\big((t+b+2)^2-(y-x)^2\big)^{-2} \mathsf{F}\Big(\tfrac{3-\sqrt{\delta}}{2},\tfrac{3-\sqrt{\delta}}{2};2; z\Big)  \\  & \qquad \quad +  \big(-\tfrac{\mu}{2}+\tfrac{1-\sqrt{\delta}}{2}\big)(1+t)^{-1}  \mathsf{F}\Big(\tfrac{1-\sqrt{\delta}}{2},\tfrac{1-\sqrt{\delta}}{2};1; z\Big)\\   
 & \qquad \quad + (\sqrt{\delta}-1)   \left((t+b+2)^2-(y-x)^2\right)^{-1} (t+b+2) \,  \mathsf{F}\Big(\tfrac{1-\sqrt{\delta}}{2},\tfrac{1-\sqrt{\delta}}{2};1; z\Big)  \bigg].
\end{align*} As $z(t,x;b,x\pm(t-b))=0$ and $\mathsf{F}(\alpha,\beta;\gamma;0)=1$, then,
\begin{align}
& \frac{\partial E}{\partial t}(t,x;b,y;\mu,\nu^2) \Big|_{y=x\pm(t-b)} \notag \\ 
 & \qquad = 2^{\sqrt{\delta}-1}(1+b)^{\frac{\mu}{2}-1} (1+t)^{-\frac{\mu}{2}-1} \bigg[ 2^{-3}(1-\sqrt{\delta})^2(t-b)  +  \big(-\tfrac{\mu}{2}+\tfrac{1-\sqrt{\delta}}{2}\big)(1+b)   +2^{-2} (\sqrt{\delta}-1)   (t+b+2) \bigg]. \label{dE/dt y=x pm(t-b)}
\end{align} On the other hand,
\begin{align*}
\frac{\partial E}{\partial x}(t,x;b,y;\mu,\nu^2) 
& =  (1+b)^{\frac{\mu}{2}+\frac{1-\sqrt{\delta}}{2}} (1+t)^{-\frac{\mu}{2}+\frac{1-\sqrt{\delta}}{2}}  \partial_x \Big( \left((t+b+2)^2-(y-x)^2\right)^{\frac{\sqrt{\delta}-1}{2}}  \mathsf{F}\Big(\tfrac{1-\sqrt{\delta}}{2},\tfrac{1-\sqrt{\delta}}{2};1; z\Big)\Big) \\
& = (1+b)^{\frac{\mu}{2}+\frac{1-\sqrt{\delta}}{2}} (1+t)^{-\frac{\mu}{2}+\frac{1-\sqrt{\delta}}{2}}  \left((t+b+2)^2-(y-x)^2\right)^{\frac{\sqrt{\delta}-1}{2}} \\
 & \quad \times  \bigg[ \mathsf{F}_{z}\Big(\tfrac{1-\sqrt{\delta}}{2},\tfrac{1-\sqrt{\delta}}{2};1; z\Big) \, \frac{\partial z}{\partial x} + (\sqrt{\delta}-1)   \left((t+b+2)^2-(y-x)^2\right)^{-1} (y-x) \,  \mathsf{F}\Big(\tfrac{1-\sqrt{\delta}}{2},\tfrac{1-\sqrt{\delta}}{2};1; z\Big)  \bigg] \\
 & = (1+b)^{\frac{\mu}{2}+\frac{1-\sqrt{\delta}}{2}} (1+t)^{-\frac{\mu}{2}+\frac{1-\sqrt{\delta}}{2}}  \left((t+b+2)^2-(y-x)^2\right)^{\frac{\sqrt{\delta}-1}{2}} \\
 & \quad \times  \bigg[ 2(1-\sqrt{\delta})^2(1+b)(1+t) \left((t+b+2)^2-(y-x)^2\right)^{-2} (y-x)\, \mathsf{F}\Big(\tfrac{3-\sqrt{\delta}}{2},\tfrac{3-\sqrt{\delta}}{2};2; z\Big)  \\ & \quad \qquad + (\sqrt{\delta}-1)   \left((t+b+2)^2-(y-x)^2\right)^{-1} (y-x) \,  \mathsf{F}\Big(\tfrac{1-\sqrt{\delta}}{2},\tfrac{1-\sqrt{\delta}}{2};1; z\Big)  \bigg]
\end{align*} implies
\begin{align}
& \frac{\partial E}{\partial x}(t,x;b,y;\mu,\nu^2) \Big|_{y=x\pm(t-b)}  = \pm 2^{\sqrt{\delta}-1}(1+b)^{\frac{\mu}{2}-1} (1+t)^{-\frac{\mu}{2}-1} \bigg[ 2^{-3}(1-\sqrt{\delta})^2(t-b)+ 2^{-2} (\sqrt{\delta}-1)(t-b) \bigg]. \label{dE/dx y=x pm(t-b)}
\end{align} Consequently, combining \eqref{dE/dt y=x pm(t-b)} and \eqref{dE/dx y=x pm(t-b)}, we have
\begin{align*}
& \bigg[ \frac{\partial E}{\partial t}(t,x;b,y;\mu,\nu^2) \mp \frac{\partial E}{\partial x}(t,x;b,y;\mu,\nu^2) \bigg]_{y=x\pm(t-b)} \\ & \qquad = 2^{\sqrt{\delta}-1}(1+b)^{\frac{\mu}{2}-1} (1+t)^{-\frac{\mu}{2}-1} \bigg[ 2^{-3}(1-\sqrt{\delta})^2(t-b)  +  \big(-\tfrac{\mu}{2}+\tfrac{1-\sqrt{\delta}}{2}\big)(1+b)   +2^{-2} (\sqrt{\delta}-1)   (t+b+2) \\ & \qquad \qquad    - 2^{-3}(1-\sqrt{\delta})^2(t-b)- 2^{-2} (\sqrt{\delta}-1)(t-b)  \bigg] \\
 \\ & \qquad = -2^{\sqrt{\delta}-2}\mu (1+b)^{\frac{\mu}{2}} (1+t)^{-\frac{\mu}{2}-1}, 
\end{align*} which are the desired estimates. The proof is complete.
\end{proof}

\subsection{The inhomogeneous problem with vanishing data}
\label{Subsection inhom 1d CP vanishing data}

In this subsection we prove that $u^{\ih}$ is a solution of \eqref{inhomog CP vanishing data 1d}. Let us determine first the time derivative of $u^{\ih}$. Hence,
\begin{align*}
\partial_t u^{\ih}(t,x) & = \frac{1}{2^{\sqrt{\delta}}} \int_0^t \partial_t \bigg( \int_{x-t+b}^{x+t-b} f(b,y) E(t,x;b,y;\mu,\nu^2) \, \mathrm{d}y\, \bigg) \mathrm{d}b \\
& = \frac{1}{2^{\sqrt{\delta}}} \int_0^t  \int_{x-t+b}^{x+t-b} f(b,y)\frac{ \partial E}{\partial t}(t,x;b,y;\mu,\nu^2) \, \mathrm{d}y\,  \mathrm{d}b + \frac{1}{2^{\sqrt{\delta}}} \int_0^t   f(b,x+t-b)  E(t,x;b,x+t-b;\mu,\nu^2) \, \mathrm{d}b \\ & \qquad + \frac{1}{2^{\sqrt{\delta}}} \int_0^t   f(b,x-t+b)  E(t,x;b,x-t+b;\mu,\nu^2) \, \mathrm{d}b.
\end{align*} Therefore, it follows immediately that $ u^{\ih}(0,x)= \partial_t u^{\ih}(0,x)=0$. It remains to prove that $u^{\ih}$ solves the differential equation. 
Moreover,
\begin{align}
E(t,x;b,x\pm (t-b) ;\mu,\nu^2) & = (1+t)^{-\frac{\mu}{2}+\frac{1-\sqrt{\delta}}{2}} (1+b)^{\frac{\mu}{2}+\frac{1-\sqrt{\delta}}{2}} \big((t+b+2)^2-(t-b)^2\big)^{\frac{\sqrt{\delta}-1}{2}} \mathsf{F}\left(\tfrac{1-\sqrt{\delta}}{2},\tfrac{1-\sqrt{\delta}}{2};1;0\right) \notag \\
& = (1+t)^{-\frac{\mu}{2}+\frac{1-\sqrt{\delta}}{2}} (1+b)^{\frac{\mu}{2}+\frac{1-\sqrt{\delta}}{2}} \big(4(1+t)(1+b)\big)^{\frac{\sqrt{\delta}-1}{2}}\notag\\
&= 2^{\sqrt{\delta}-1}(1+t)^{-\frac{\mu}{2}} (1+b)^{\frac{\mu}{2}} \label{value E for y=x pm(t-b)}
\end{align} implies
\begin{align}
\partial_t u^{\ih}(t,x) & = \frac{1}{2^{\sqrt{\delta}}} \int_0^t  \int_{x-t+b}^{x+t-b} f(b,y)\frac{ \partial E}{\partial t}(t,x;b,y;\mu,\nu^2) \, \mathrm{d}y\,  \mathrm{d}b \notag \\ & \qquad + \frac{1}{2} \int_0^t   \big[ f(b,x+t-b)+ f(b,x-t+b)  \big] (1+t)^{-\frac{\mu}{2}}(1+b)^{\frac{\mu}{2}} \, \mathrm{d}b. \label{u ih 1st t der}
\end{align}
We may calculate now the second order derivative with respect to $t$. Differentiating the last relation, we get 
\begin{align}
\partial_t^2 u^{\ih}(t,x) & = \frac{1}{2^{\sqrt{\delta}}} \int_0^t \partial_t \bigg( \int_{x-t+b}^{x+t-b} f(b,y)\frac{ \partial E}{\partial t}(t,x;b,y;\mu,\nu^2) \, \mathrm{d}y\, \bigg)  \mathrm{d}b \notag\\ 
& \qquad+ \frac{1}{2} \int_0^t \partial_t \Big(  \big[ f(b,x+t-b)+ f(b,x-t+b)  \big] (1+t)^{-\frac{\mu}{2}}(1+b)^{\frac{\mu}{2}}\Big) \, \mathrm{d}b +f(t,x) \notag \\
& = \frac{1}{2^{\sqrt{\delta}}} \int_0^t  \int_{x-t+b}^{x+t-b} f(b,y)\frac{ \partial^2 E}{\partial t^2}(t,x;b,y;\mu,\nu^2) \, \mathrm{d}y\,   \mathrm{d}b \notag\\
 & \qquad+  \frac{1}{2^{\sqrt{\delta}}} \int_0^t \bigg[ f(b,y)\frac{ \partial E}{\partial t}(t,x;b,y;\mu,\nu^2)\bigg]_{y=x+t-b} \,   \mathrm{d}b  +  \frac{1}{2^{\sqrt{\delta}}} \int_0^t \bigg[ f(b,y)\frac{ \partial E}{\partial t}(t,x;b,y;\mu,\nu^2)\bigg]_{y=x-t+b} \,   \mathrm{d}b \notag \\
  & \qquad -\frac{\mu}{4} \int_0^t \big[ f(b,x+t-b)+ f(b,x-t+b)  \big] (1+t)^{-\frac{\mu}{2}-1}(1+b)^{\frac{\mu}{2}} \, \mathrm{d}b \notag \\ 
  & \qquad+ \frac{1}{2} \int_0^t  \bigg[ \frac{\partial f}{\partial x}(b,x+t-b)- \frac{\partial f}{\partial x} f(b,x-t+b)  \bigg] (1+t)^{-\frac{\mu}{2}}(1+b)^{\frac{\mu}{2}}\, \mathrm{d}b +f(t,x). \label{u ih 2nd t der}
\end{align} The next step is to calculate the derivative of order two with respect to $x$ of $u^{\ih}$. Let us begin with the derivative of order one:
\begin{align*}
\partial_x u^{\ih}(t,x) &  = \frac{1}{2^{\sqrt{\delta}}} \int_0^t  \int_{x-t+b}^{x+t-b} f(b,y)\frac{ \partial E}{\partial x}(t,x;b,y;\mu,\nu^2) \, \mathrm{d}y\,  \mathrm{d}b + \frac{1}{2^{\sqrt{\delta}}} \int_0^t   f(b,x+t-b)  E(t,x;b,x+t-b;\mu,\nu^2) \, \mathrm{d}b \\ & \qquad - \frac{1}{2^{\sqrt{\delta}}} \int_0^t   f(b,x-t+b)  E(t,x;b,x-t+b;\mu,\nu^2) \, \mathrm{d}b \\
&  = \frac{1}{2^{\sqrt{\delta}}} \int_0^t  \int_{x-t+b}^{x+t-b} f(b,y)\frac{ \partial E}{\partial x}(t,x;b,y;\mu,\nu^2) \, \mathrm{d}y\,  \mathrm{d}b \\ & \qquad + \frac{1}{2} \int_0^t  \big[f(b,x+t-b)- f(b,x-t+b)\big]  (1+t)^{-\frac{\mu}{2}} (1+b)^{\frac{\mu}{2}} \, \mathrm{d}b,
\end{align*} where in the second equality we used again \eqref{value E for y=x pm(t-b)}. A further derivation with respect to $x$ of the last expression provides
\begin{align}
\partial_x^2 u^{\ih}(t,x) 
&  = \frac{1}{2^{\sqrt{\delta}}} \int_0^t  \int_{x-t+b}^{x+t-b} f(b,y)\frac{ \partial^2 E}{\partial x^2}(t,x;b,y;\mu,\nu^2) \, \mathrm{d}y\,  \mathrm{d}b \notag \\ 
& \qquad+  \frac{1}{2^{\sqrt{\delta}}} \int_0^t \bigg[ f(b,y)\frac{ \partial E}{\partial x}(t,x;b,y;\mu,\nu^2)\bigg]_{y=x+t-b} \,   \mathrm{d}b   -  \frac{1}{2^{\sqrt{\delta}}} \int_0^t  \bigg[f(b,y)\frac{ \partial E}{\partial x}(t,x;b,y;\mu,\nu^2)\bigg]_{y=x-t+b} \,   \mathrm{d}b  \notag \\
 & \qquad+ \frac{1}{2} \int_0^t  \bigg[ \frac{\partial f}{\partial x}(b,x+t-b)- \frac{\partial f}{\partial x} f(b,x-t+b)  \bigg] (1+t)^{-\frac{\mu}{2}}(1+b)^{\frac{\mu}{2}}\, \mathrm{d}b.  \label{u ih 2nd x der}
\end{align}
Combining \eqref{inhomogeneous sol 1d}, \eqref{u ih 1st t der}, \eqref{u ih 2nd t der} and \eqref{u ih 2nd x der}, we arrive at
\begin{align*}
& \partial_t^2 u^{\ih}(t,x) -\partial_x^2 u^{\ih}(t,x) +\tfrac{\mu}{1+t} \partial_t u^{\ih}(t,x) +\tfrac{\nu^2}{(1+t)^2} u^{\ih}(t,x) \\ & \qquad =  \frac{1}{2^{\sqrt{\delta}}} \int_0^t  \int_{x-t+b}^{x+t-b} f(b,y)\bigg[\frac{ \partial^2 }{\partial t^2}-\frac{ \partial^2 }{\partial x^2} +\frac{\mu}{1+t}\frac{ \partial }{\partial t}+\frac{\nu^2}{(1+t)^2}\bigg]E(t,x;b,y;\mu,\nu^2) \, \mathrm{d}y\,  \mathrm{d}b + f(t,x) \\
& \qquad \quad +  \frac{1}{2^{\sqrt{\delta}}} \int_0^t \bigg[ f(b,y)\bigg(\frac{ \partial E}{\partial t}(t,x;b,y;\mu,\nu^2)-\frac{ \partial E}{\partial x}(t,x;b,y;\mu,\nu^2)+2^{\sqrt{\delta}-2}\mu (1+t)^{-\frac{\mu}{2}-1}(1+b)^{\frac{\mu}{2}}\bigg)\bigg]_{y=x+t-b} \,   \mathrm{d}b  \\
& \qquad \quad +  \frac{1}{2^{\sqrt{\delta}}} \int_0^t \bigg[ f(b,y)\bigg(\frac{ \partial E}{\partial t}(t,x;b,y;\mu,\nu^2)+\frac{ \partial E}{\partial x}(t,x;b,y;\mu,\nu^2)+2^{\sqrt{\delta}-2}\mu (1+t)^{-\frac{\mu}{2}-1}(1+b)^{\frac{\mu}{2}}\bigg)\bigg]_{y=x-t+b} \,   \mathrm{d}b.  
\end{align*} However, by Proposition \ref{Prop E is a fund sol} and Lemma \ref{Lemma dt E-dx E + ... =0} it follows that the all integrands on the right hand side of the last equality vanish. Consequently,
\begin{align*}
& \partial_t^2 u^{\ih}(t,x) -\partial_x^2 u^{\ih}(t,x) +\tfrac{\mu}{1+t} \partial_t u^{\ih}(t,x) +\tfrac{\nu^2}{(1+t)^2} u^{\ih}(t,x) = f(t,x).
\end{align*} So, we proved that $u^{\ih}$ solves \eqref{inhomog CP vanishing data 1d}.

\subsection{The homogeneous problem with nontrivial data}
\label{Subsection hom 1d CP}

In this subsection we will prove that $u^{\h}$ defined in \eqref{homogeneous sol 1d} is a solution of \eqref{homog CP 1d}. For this purpose, we consider the function $w=w(t,x)\doteq u(t,x)-u_0(x)-tu_1(x)$. If $u$ solves \eqref{homogeneous sol 1d}, then, $w$ solves \eqref{inhomog CP vanishing data 1d} with $$f=f(t,x)=u_0''(x)+tu_1''(x)-\left(\frac{\nu^2}{(1+t)^2}u_0(x)+\frac{\mu}{1+t}u_1(x)+\frac{\nu^2\, t}{(1+t)^2}u_1(x)\right).$$ Therefore, according to the representation formula derived in Subsection \ref{Subsection inhom 1d CP vanishing data}, we obtain 
\begin{align*}
w(t,x) & = \frac{1}{2^{\sqrt{\delta}}} \int_0^t \int_{x-t+b}^{x+t-b} \left[u_0''(y)+bu_1''(y)-\left(\frac{\nu^2}{(1+b)^2}u_0(y)+\frac{\mu}{1+b}u_1(y)+\frac{\nu^2\, b}{(1+b)^2}u_1(y)\right)\right] E(t,x;b,y;\mu,\nu^2) \, \mathrm{d}y\, \mathrm{d}b \\ & \doteq  I_1+I_2+I_3+I_4+I_5.
\end{align*}
Now we will manipulate $I_1,I_2$ in order to get the cancellation of some terms in the expression of $w$ and, then, $u$. Let us begin with $I_1 = \frac{1}{2^{\sqrt{\delta}}} \int_0^t \int_{x-t+b}^{x+t-b} u_0''(y) E(t,x;b,y;\mu,\nu^2) \, \mathrm{d}y\, \mathrm{d}b$. Using  twice integration by parts and Corollary \ref{Cor E is a fund sol adj eq}, we find
\begin{align*}
& \int_{x-t+b}^{x+t-b} u_0''(y) E(t,x;b,y;\mu,\nu^2) \, \mathrm{d}y \\
& \quad = \bigg[u'_0(y)E(t,x;b,y;\mu,\nu^2)-u_0(y)\frac{\partial E}{\partial y}(t,x;b,y;\mu,\nu^2)\bigg]^{y=x+t-b}_{y=x-t+b} + \int_{x-t+b}^{x+t-b} u_0(y)\frac{\partial^2 E}{\partial y^2} (t,x;b,y;\mu,\nu^2) \, \mathrm{d}y \\
& \quad = \bigg[u'_0(y)E(t,x;b,y;\mu,\nu^2)-u_0(y)\frac{\partial E}{\partial y}(t,x;b,y;\mu,\nu^2)\bigg]^{y=x+t-b}_{y=x-t+b} +\int_{x-t+b}^{x+t-b} u_0(y)\frac{\partial^2 E}{\partial b^2} (t,x;b,y;\mu,\nu^2) \, \mathrm{d}y \\ 
&  \qquad + \int_{x-t+b}^{x+t-b} u_0(y)\bigg[-\frac{\mu}{1+b}\frac{\partial E}{\partial b}(t,x;b,y;\mu,\nu^2) + \frac{\mu+\nu^2}{(1+b)^2}E(t,x;b,y;\mu,\nu^2)  \bigg] \, \mathrm{d}y.
\end{align*} Hence,
\begin{align*}
I_1 & = \frac{1}{2^{\sqrt{\delta}}} \int_0^t  \bigg[u'_0(y)E(t,x;b,y;\mu,\nu^2)-u_0(y)\frac{\partial E}{\partial y}(t,x;b,y;\mu,\nu^2)\bigg]^{y=x+t-b}_{y=x-t+b} \mathrm{d}b \\
& \qquad +\frac{1}{2^{\sqrt{\delta}}} \int_0^t  \int_{x-t+b}^{x+t-b} u_0(y)\frac{\partial^2 E}{\partial b^2} (t,x;b,y;\mu,\nu^2) \, \mathrm{d}y \, \mathrm{d}b\\ 
&  \qquad - \frac{1}{2^{\sqrt{\delta}}} \int_0^t  \int_{x-t+b}^{x+t-b} u_0(y)\frac{\mu}{1+b}\frac{\partial E}{\partial b}(t,x;b,y;\mu,\nu^2)  \, \mathrm{d}y\, \mathrm{d}b \\
& \qquad + \frac{1}{2^{\sqrt{\delta}}} \int_0^t  \int_{x-t+b}^{x+t-b} u_0(y) \frac{\mu+\nu^2}{(1+b)^2}E(t,x;b,y;\mu,\nu^2)   \, \mathrm{d}y\, \mathrm{d}b \doteq J_1+J_2+J_3+J_4.
\end{align*} Let us rewrite $J_2,J_3$ in a more suitable way. By using Fubini's theorem and integration by parts, we get
\begin{align*}
J_2 & = \frac{1}{2^{\sqrt{\delta}}} \int_0^t  \int_{x-t+b}^{x+t-b} u_0(y)\frac{\partial^2 E}{\partial b^2} (t,x;b,y;\mu,\nu^2) \, \mathrm{d}y \, \mathrm{d}b = \frac{1}{2^{\sqrt{\delta}}} \int_{x-t}^{x+t}  u_0(y) \int_{0}^{t-|x-y|} \frac{\partial^2 E}{\partial b^2} (t,x;b,y;\mu,\nu^2)\, \mathrm{d}b \, \mathrm{d}y  \\
& = \frac{1}{2^{\sqrt{\delta}}} \int_{x-t}^{x+t}  u_0(y) \bigg[\frac{\partial E}{\partial b} (t,x;b,y;\mu,\nu^2)\bigg]_{b=0}^{b=t-|x-y|}   \, \mathrm{d}y,  \\
J_3 &= - \frac{1}{2^{\sqrt{\delta}}} \int_0^t  \int_{x-t+b}^{x+t-b} u_0(y)\frac{\mu}{1+b}\frac{\partial E}{\partial b}(t,x;b,y;\mu,\nu^2)  \, \mathrm{d}y\, \mathrm{d}b = - \frac{1}{2^{\sqrt{\delta}}} \int_{x-t}^{x+t}  u_0(y) \int_{0}^{t-|x-y|}  \frac{\mu}{1+b}\frac{\partial E}{\partial b}(t,x;b,y;\mu,\nu^2) \, \mathrm{d}b  \, \mathrm{d}y \\
&= -\frac{1}{2^{\sqrt{\delta}}} \int_{x-t}^{x+t}  \! u_0(y)\bigg[\frac{\mu}{1+b}E(t,x;b,y;\mu,\nu^2)\bigg]_{b=0}^{b=t-|x-y|}   \mathrm{d}y + \frac{1}{2^{\sqrt{\delta}}} \int_{x-t}^{x+t}  u_0(y) \int_{0}^{t-|x-y|} \! \frac{\partial }{\partial b}\bigg(\frac{\mu}{1+b}\bigg) E(t,x;b,y;\mu,\nu^2) \, \mathrm{d}b  \, \mathrm{d}y \\
&= -\frac{1}{2^{\sqrt{\delta}}} \int_{x-t}^{x+t}  \! u_0(y)\bigg[\frac{\mu}{1+b}E(t,x;b,y;\mu,\nu^2)\bigg]_{b=0}^{b=t-|x-y|}   \mathrm{d}y - \frac{1}{2^{\sqrt{\delta}}} \int_0^t  \int_{x-t+b}^{x+t-b} u_0(y) \frac{\mu}{(1+b)^2} E(t,x;b,y;\mu,\nu^2) \, \mathrm{d}b  \, \mathrm{d}y.
\end{align*} In particular, from the last relation we see that there is a cancellation between one term in $J_3$ and another one in $J_4$.  Combining the expressions for $J_2,J_3$ that  we have just proved, we obtain
\begin{align}
I_1= J_1+J_2+J_3+J_4 = J_1+ \frac{1}{2^{\sqrt{\delta}}} \int_{x-t}^{x+t}  u_0(y) \bigg[\frac{\partial E}{\partial b} (t,x;b,y;\mu,\nu^2)-\frac{\mu}{1+b}E(t,x;b,y;\mu,\nu^2)\bigg]_{b=0}^{b=t-|x-y|}   \, \mathrm{d}y - I_3. \label{intermediate estimate I1}
\end{align}
Now, we remark that 
\begin{align*}
u'_0(x\pm (t-b))=\mp \frac{\partial}{\partial b} \big( u_0(x\pm (t-b))\big).
\end{align*} Hence, we have
\begin{align*}
& \int_0^t  \bigg[u'_0(y)E(t,x;b,y;\mu,\nu^2)\bigg]^{y=x+t-b}_{y=x-t+b} \mathrm{d}b \\
& \quad = - \int_0^t  \bigg[\frac{\partial}{\partial b} \big( u_0(x+t-b)\big)E(t,x;b,x+t-b;\mu,\nu^2)+\frac{\partial}{\partial b} \big( u_0(x-t+b)\big)E(t,x;b,x-t+b;\mu,\nu^2)\bigg] \mathrm{d}b \\
& \quad = -  \bigg[ u_0(x+t-b)E(t,x;b,x+t-b;\mu,\nu^2)+ u_0(x-t+b) E(t,x;b,x-t+b;\mu,\nu^2)\bigg]_{b=0}^{b=t}   \\
& \qquad + \int_0^t  \bigg[ u_0(x+t-b) \frac{\partial }{\partial b} \Big(E(t,x;b,x+t-b;\mu,\nu^2)\Big) + u_0(x-t+b) \frac{\partial }{\partial b} \Big(E(t,x;b,x-t+b;\mu,\nu^2)\Big)\bigg] \mathrm{d}b \\
& \quad = -  2 \, u_0(x)E(t,x;t,x;\mu,\nu^2)+ \bigg[ u_0(x+t)E(t,x;0,x+t;\mu,\nu^2)+ u_0(x-t) E(t,x;0,x-t;\mu,\nu^2)\bigg]  \\
& \qquad + \int_0^t  \bigg[ u_0(y) \bigg(\frac{\partial E}{\partial b} (t,x;b,y;\mu,\nu^2)-\frac{\partial E}{\partial y} (t,x;b,y;\mu,\nu^2)\bigg)\bigg]_{y=x+t-b} \mathrm{d}b  \\
& \qquad+ \int_0^t  \bigg[ u_0(y) \bigg(\frac{\partial E}{\partial b} (t,x;b,y;\mu,\nu^2)+\frac{\partial E}{\partial y} (t,x;b,y;\mu,\nu^2)\bigg)\bigg]_{y=x-t+b} \mathrm{d}b.
\end{align*} Since 
\begin{align}
E(t,x;t,x;\mu,\nu^2) & = (1+t)^{-\frac{\mu}{2}+\frac{1-\sqrt{\delta}}{2}} (1+t)^{\frac{\mu}{2}+\frac{1-\sqrt{\delta}}{2}} \big(2^2(1+t)^2\big)^{\frac{\sqrt{\delta}-1}{2}} \mathsf{F}\left(\tfrac{1-\sqrt{\delta}}{2},\tfrac{1-\sqrt{\delta}}{2};1;0\right)=2^{\sqrt{\delta}-1}, \label{value E(t,x:t,x)}\\
E(t,x;0,x\pm t;\mu,\nu^2) & = (1+t)^{-\frac{\mu}{2}+\frac{1-\sqrt{\delta}}{2}} \big((t+2)^2-t^2\big)^{\frac{\sqrt{\delta}-1}{2}} \mathsf{F}\left(\tfrac{1-\sqrt{\delta}}{2},\tfrac{1-\sqrt{\delta}}{2};1;0\right)=2^{\sqrt{\delta}-1}(1+t)^{-\frac{\mu}{2}}, \notag
\end{align} plugging the right hand side of the last chain of equalities in $J_1$, we get 
\begin{align}
J_1 &= \frac{1}{2^{\sqrt{\delta}}} \int_0^t  \bigg[u'_0(y)E(t,x;b,y;\mu,\nu^2)-u_0(y)\frac{\partial E}{\partial y}(t,x;b,y;\mu,\nu^2)\bigg]^{y=x+t-b}_{y=x-t+b} \mathrm{d}b \notag\\
& = -  u_0(x)+ \frac{1}{2}(1+t)^{-\frac{\mu}{2}}\big[ u_0(x+t)+ u_0(x-t) \big] \notag \\
& \qquad + \frac{1}{2^{\sqrt{\delta}}} \int_0^t  \bigg[ u_0(y) \bigg(\frac{\partial E}{\partial b} (t,x;b,y;\mu,\nu^2)-2\frac{\partial E}{\partial y} (t,x;b,y;\mu,\nu^2)\bigg)\bigg]_{y=x+t-b} \mathrm{d}b \notag \\
& \qquad+ \frac{1}{2^{\sqrt{\delta}}} \int_0^t  \bigg[ u_0(y) \bigg(\frac{\partial E}{\partial b} (t,x;b,y;\mu,\nu^2)+2\frac{\partial E}{\partial y} (t,x;b,y;\mu,\nu^2)\bigg)\bigg]_{y=x-t+b} \mathrm{d}b. \label{formula J1}
\end{align} Before plugging this expression in \eqref{intermediate estimate I1}, let us rewrite the integral term in the extreme $b=t-|x-y|$ on the right hand side of \eqref{intermediate estimate I1} in a more convenient way, namely,
\begin{align}
& \int_{x-t}^{x+t}  u_0(y) \bigg[\frac{\partial E}{\partial b} (t,x;b,y;\mu,\nu^2)-\frac{\mu}{1+b}E(t,x;b,y;\mu,\nu^2)\bigg]_{b=t-|x-y|}   \, \mathrm{d}y \notag \\
&  \qquad = \int_{x}^{x+t}  u_0(y) \bigg[\frac{\partial E}{\partial b} (t,x;b,y;\mu,\nu^2)-\frac{\mu}{1+b}E(t,x;b,y;\mu,\nu^2)\bigg]_{b=t+x-y}   \, \mathrm{d}y \notag \\
&  \qquad \quad  + \int_{x-t}^{x}  u_0(y) \bigg[\frac{\partial E}{\partial b} (t,x;b,y;\mu,\nu^2)-\frac{\mu}{1+b}E(t,x;b,y;\mu,\nu^2)\bigg]_{b=t-x+y}   \, \mathrm{d}y \notag \\
&  \qquad = \int_{0}^{t}  u_0(x+t-b) \bigg[\frac{\partial E}{\partial b} (t,x;b,y;\mu,\nu^2)-\frac{\mu}{1+b}E(t,x;b,y;\mu,\nu^2)\bigg]_{y=x+t-b}   \, \mathrm{d}y \notag \\
&  \qquad \quad  + \int_{0}^{t}  u_0(x-t+b) \bigg[\frac{\partial E}{\partial b} (t,x;b,y;\mu,\nu^2)-\frac{\mu}{1+b}E(t,x;b,y;\mu,\nu^2)\bigg]_{y=x-t+b}   \, \mathrm{d}y. \label{estimate J2+J3}
\end{align}  Combining \eqref{intermediate estimate I1}, \eqref{formula J1} and \eqref{estimate J2+J3}, it follows
\begin{align*}
I_1+I_3 &=  -  u_0(x)+ \frac{1}{2}(1+t)^{-\frac{\mu}{2}}\big[ u_0(x+t)+ u_0(x-t) \big] \\
& \qquad + \frac{1}{2^{\sqrt{\delta}}} \int_{x-t}^{x+t}  u_0(y) \bigg[-\frac{\partial E}{\partial b} (t,x;b,y;\mu,\nu^2)+\frac{\mu}{1+b}E(t,x;b,y;\mu,\nu^2)\bigg]_{b=0} \mathrm{d}y \\
& \qquad + \frac{1}{2^{\sqrt{\delta}}} \int_0^t  \bigg[ u_0(y) \bigg(2\frac{\partial E}{\partial b} (t,x;b,y;\mu,\nu^2)-2\frac{\partial E}{\partial y} (t,x;b,y;\mu,\nu^2)-\frac{\mu}{1+b}E(t,x;b,y;\mu,\nu^2)\bigg)\bigg]_{y=x+t-b} \mathrm{d}b  \\
& \qquad+ \frac{1}{2^{\sqrt{\delta}}}\int_0^t  \bigg[ u_0(y) \bigg(2\frac{\partial E}{\partial b} (t,x;b,y;\mu,\nu^2)+2\frac{\partial E}{\partial y} (t,x;b,y;\mu,\nu^2)-\frac{\mu}{1+b}E(t,x;b,y;\mu,\nu^2)\bigg)\bigg]_{y=x-t+b} \mathrm{d}b.
\end{align*} Next, we shall prove that the  functions that multiply $u_0(y)$ in the last two integrals in the previous formula for $I_1+I_3$ are identically zero on the domain of integration.
Using the identities
\begin{align*}
\frac{\partial z}{\partial b}(t,x;b,y) & =\frac{4(1+t)\big[(y-x)^2-(t-b)(t+b+2)\big]}{\big[(t+b+2)^2-(y-x)^2\big]^2}, \\
\frac{\partial z}{\partial y}(t,x;b,y) & =-\frac{8(y-x)(1+t)(1+b)}{\big[(t+b+2)^2-(y-x)^2\big]^2},
\end{align*} and the recursive relation for the derivative of a hypergeometric function, we get 
\begin{align*}
\frac{\partial E}{\partial b}(t,x;b,y;\mu,\nu^2) &= (1+t)^{-\frac{\mu}{2}+\frac{1-\sqrt{\delta}}{2}} (1+b)^{\frac{\mu}{2}+\frac{1-\sqrt{\delta}}{2}} \big((t+b+2)^2-(y-x)^2 \big)^{\frac{\sqrt{\delta}-1}{2}} \\
& \quad \times \bigg[\mathsf{F}_z\left(\tfrac{1-\sqrt{\delta}}{2}, \tfrac{1-\sqrt{\delta}}{2}; 1;z \right) \frac{\partial z}{\partial b}+\big(\tfrac{\mu}{2}+\tfrac{1-\sqrt{\delta}}{2}\big) (1+b)^{-1} \mathsf{F}\left(\tfrac{1-\sqrt{\delta}}{2}, \tfrac{1-\sqrt{\delta}}{2}; 1;z \right)
\\
& \quad \qquad + (\sqrt{\delta}-1)\big((t+b+2)^2-(y-x)^2 \big)^{-1} (t+b+2) \, \mathsf{F}\left(\tfrac{1-\sqrt{\delta}}{2}, \tfrac{1-\sqrt{\delta}}{2}; 1;z \right)\bigg] \\
&= (1+t)^{-\frac{\mu}{2}+\frac{1-\sqrt{\delta}}{2}} (1+b)^{\frac{\mu}{2}+\frac{1-\sqrt{\delta}}{2}} \big((t+b+2)^2-(y-x)^2 \big)^{\frac{\sqrt{\delta}-1}{2}} \\
& \quad \times \bigg[(1-\sqrt{\delta})^2(1+t)\big((y-x)^2-(t-b)(t+b+2)\big) \big((t+b+2)^2-(y-x)^2 \big)^{-2}\mathsf{F}\left(\tfrac{3-\sqrt{\delta}}{2}, \tfrac{3-\sqrt{\delta}}{2}; 2;z \right) \\ 
& \quad \qquad +\big(\tfrac{\mu}{2}+\tfrac{1-\sqrt{\delta}}{2}\big) (1+b)^{-1} \mathsf{F}\left(\tfrac{1-\sqrt{\delta}}{2}, \tfrac{1-\sqrt{\delta}}{2}; 1;z \right)
\\
& \quad \qquad + (\sqrt{\delta}-1)\big((t+b+2)^2-(y-x)^2 \big)^{-1} (t+b+2) \, \mathsf{F}\left(\tfrac{1-\sqrt{\delta}}{2}, \tfrac{1-\sqrt{\delta}}{2}; 1;z \right)\bigg], \\
\frac{\partial E}{\partial y}(t,x;b,y;\mu,\nu^2) &= (1+t)^{-\frac{\mu}{2}+\frac{1-\sqrt{\delta}}{2}} (1+b)^{\frac{\mu}{2}+\frac{1-\sqrt{\delta}}{2}} \big((t+b+2)^2-(y-x)^2 \big)^{\frac{\sqrt{\delta}-1}{2}} \\
& \quad \times \bigg[\mathsf{F}_z\left(\tfrac{1-\sqrt{\delta}}{2}, \tfrac{1-\sqrt{\delta}}{2}; 1;z \right) \frac{\partial z}{\partial y}
 - (\sqrt{\delta}-1)\big((t+b+2)^2-(y-x)^2 \big)^{-1} (y-x) \, \mathsf{F}\left(\tfrac{1-\sqrt{\delta}}{2}, \tfrac{1-\sqrt{\delta}}{2}; 1;z \right)\bigg] \\
&= (1+t)^{-\frac{\mu}{2}+\frac{1-\sqrt{\delta}}{2}} (1+b)^{\frac{\mu}{2}+\frac{1-\sqrt{\delta}}{2}} \big((t+b+2)^2-(y-x)^2 \big)^{\frac{\sqrt{\delta}-1}{2}} \\
& \quad \times \bigg[-2(1-\sqrt{\delta})^2(y-x)(1+t)(1+b)\big((t+b+2)^2-(y-x)^2 \big)^{-2}\mathsf{F}\left(\tfrac{3-\sqrt{\delta}}{2}, \tfrac{3-\sqrt{\delta}}{2}; 2;z \right) 
\\
& \quad \qquad - (\sqrt{\delta}-1)\big((t+b+2)^2-(y-x)^2 \big)^{-1} (y-x) \, \mathsf{F}\left(\tfrac{1-\sqrt{\delta}}{2}, \tfrac{1-\sqrt{\delta}}{2}; 1;z \right)\bigg], \\
\frac{ \mu}{1+ b}E(t,x;b,y;\mu,\nu^2) & = \mu (1+t)^{-\frac{\mu}{2}+\frac{1-\sqrt{\delta}}{2}} (1+b)^{\frac{\mu}{2}-1+\frac{1-\sqrt{\delta}}{2}} \big((t+b+2)^2-(y-x)^2 \big)^{\frac{\sqrt{\delta}-1}{2}}\mathsf{F}\left(\tfrac{1-\sqrt{\delta}}{2}, \tfrac{1-\sqrt{\delta}}{2}; 1;z \right).
\end{align*} Evaluating these functions in $y=x\pm (t-b)$, we get 
\begin{align*}
\frac{\partial E}{\partial b}(t,x;b,y;\mu,\nu^2)\Big|_{y=x\pm (t-b)} & = 2^{\sqrt{\delta}-1}(1+t)^{-\frac{\mu}{2}-1} (1+b)^{\frac{\mu}{2}-1} \\ & \quad \times\Big[-2^{-3}(1-\sqrt{\delta})^{2}(t-b)+\big(\tfrac{\mu}{2}+\tfrac{1-\sqrt{\delta}}{2}\big)(1+t)+2^{-2}(\sqrt{\delta}-1)(t+b+2)\Big], \\
\frac{\partial E}{\partial y}(t,x;b,y;\mu,\nu^2)\Big|_{y=x\pm (t-b)} & = \mp \, 2^{\sqrt{\delta}-1}(1+t)^{-\frac{\mu}{2}-1} (1+b)^{\frac{\mu}{2}-1} \Big[2^{-3}(1-\sqrt{\delta})^{2}(t-b)+2^{-2}(\sqrt{\delta}-1)(t-b)\Big], \\
\frac{ \mu}{1+ b}E(t,x;b,y;\mu,\nu^2)\Big|_{y=x\pm (t-b)} & =\, 2^{\sqrt{\delta}-1}\mu (1+t)^{-\frac{\mu}{2}} (1+b)^{\frac{\mu}{2}-1} .
\end{align*} Therefore,
\begin{align}
& \bigg(2\frac{\partial E}{\partial b} (t,x;b,y;\mu,\nu^2)\mp 2\frac{\partial E}{\partial y} (t,x;b,y;\mu,\nu^2)-\frac{\mu}{1+b}E(t,x;b,y;\mu,\nu^2)\bigg)_{y=x\pm (t-b)}\notag \\
& \qquad = 2^{\sqrt{\delta}-1}(1+t)^{-\frac{\mu}{2}-1} (1+b)^{\frac{\mu}{2}-1}  \Big[ 2\big(\tfrac{\mu}{2}+\tfrac{1-\sqrt{\delta}}{2}\big)(1+t)+2^{-1}(\sqrt{\delta}-1)(2t+2)-\mu (1+t)\Big]=0. \label{relation 2 dE/db -2 dE/dy}
\end{align} Summarizing, we proved
\begin{align}
I_1+I_3 &=  -  u_0(x)+ \frac{1}{2}(1+t)^{-\frac{\mu}{2}}\big[ u_0(x+t)+ u_0(x-t) \big]+ \frac{1}{2^{\sqrt{\delta}}} \int_{x-t}^{x+t}  u_0(y) \bigg[-\frac{\partial E}{\partial b} (t,x;b,y;\mu,\nu^2)\bigg]_{b=0} \mathrm{d} y \notag  \\
& \qquad + \frac{1}{2^{\sqrt{\delta}}} \int_{x-t}^{x+t} \mu\, u_0(y) E(t,x;0,y;\mu,\nu^2) \, \mathrm{d}y . \label{I1+I3}
\end{align} We consider now the term $I_2=\frac{1}{2^{\sqrt{\delta}}} \int_0^t \int_{x-t+b}^{x+t-b} b u_1''(y) E(t,x;b,y;\mu,\nu^2) \, \mathrm{d}y\, \mathrm{d}b$. We will proceed similarly as for $I_1$. Integration by parts leads to
\begin{align*}
& \int_{x-t+b}^{x+t-b} b\, u_1''(y) E(t,x;b,y;\mu,\nu^2) \, \mathrm{d}y \\
& \quad = \bigg[ u_1'(y)\, b E(t,x;b,y;\mu,\nu^2)- u_1(y)\, b \frac{\partial E}{\partial y}(t,x;b,y;\mu,\nu^2)\bigg]^{y=x+t-b}_{y=x-t+b} + \int_{x-t+b}^{x+t-b} u_1(y)\, b\frac{\partial^2 E}{\partial y^2} (t,x;b,y;\mu,\nu^2) \, \mathrm{d}y \\
& \quad = \bigg[ u_1'(y)\, b E(t,x;b,y;\mu,\nu^2)- u_1(y) \,b \frac{\partial E}{\partial y}(t,x;b,y;\mu,\nu^2)\bigg]^{y=x+t-b}_{y=x-t+b} +\int_{x-t+b}^{x+t-b} u_1(y)\,  b \frac{\partial^2 E}{\partial b^2} (t,x;b,y;\mu,\nu^2) \, \mathrm{d}y \\ 
&  \qquad + \int_{x-t+b}^{x+t-b} u_1(y)\bigg[-\frac{\mu b}{1+b}\frac{\partial E}{\partial b}(t,x;b,y;\mu,\nu^2) + \frac{(\mu+\nu^2)\, b}{(1+b)^2}E(t,x;b,y;\mu,\nu^2)  \bigg] \, \mathrm{d}y.
\end{align*} Thus,
\begin{align*}
I_2 & = \frac{1}{2^{\sqrt{\delta}}} \int_0^t  \bigg[ u_1'(y)\, b E(t,x;b,y;\mu,\nu^2)- u_1(y) \,b \frac{\partial E}{\partial y}(t,x;b,y;\mu,\nu^2)\bigg]^{y=x+t-b}_{y=x-t+b} \mathrm{d}b \\
& \qquad +\frac{1}{2^{\sqrt{\delta}}} \int_0^t  \int_{x-t+b}^{x+t-b} u_1(y)\,  b \frac{\partial^2 E}{\partial b^2} (t,x;b,y;\mu,\nu^2) \, \mathrm{d}y \, \mathrm{d}b\\ 
&  \qquad - \frac{1}{2^{\sqrt{\delta}}} \int_0^t  \int_{x-t+b}^{x+t-b} u_1(y)\frac{\mu b}{1+b}\frac{\partial E}{\partial b}(t,x;b,y;\mu,\nu^2)  \, \mathrm{d}y\, \mathrm{d}b \\
& \qquad + \frac{1}{2^{\sqrt{\delta}}} \int_0^t  \int_{x-t+b}^{x+t-b} u_1(y) \frac{(\mu+\nu^2)\, b}{(1+b)^2}E(t,x;b,y;\mu,\nu^2)   \, \mathrm{d}y\, \mathrm{d}b \doteq \widetilde{J}_1+\widetilde{J}_2+\widetilde{J}_3+\widetilde{J}_4.
\end{align*}

 Employing Fubini's theorem and integration by parts, we get
\begin{align*}
\widetilde{J}_2 & = \frac{1}{2^{\sqrt{\delta}}} \int_0^t  \int_{x-t+b}^{x+t-b} u_1(y)\,  b\frac{\partial^2 E}{\partial b^2} (t,x;b,y;\mu,\nu^2) \, \mathrm{d}y \, \mathrm{d}b = \frac{1}{2^{\sqrt{\delta}}} \int_{x-t}^{x+t}  u_1(y) \int_{0}^{t-|x-y|} b \frac{\partial^2 E}{\partial b^2} (t,x;b,y;\mu,\nu^2)\, \mathrm{d}b \, \mathrm{d}y  \\
& = \frac{1}{2^{\sqrt{\delta}}} \int_{x-t}^{x+t}  u_1(y) \bigg[b\frac{\partial E}{\partial b} (t,x;b,y;\mu,\nu^2)\bigg]_{b=0}^{b=t-|x-y|}   \, \mathrm{d}y- \frac{1}{2^{\sqrt{\delta}}} \int_{x-t}^{x+t}  u_1(y) \int_{0}^{t-|x-y|} \frac{\partial E}{\partial b} (t,x;b,y;\mu,\nu^2)\, \mathrm{d}b \, \mathrm{d}y  \\
& = \frac{1}{2^{\sqrt{\delta}}} \int_{x-t}^{x+t}  u_1(y) \bigg[b\frac{\partial E}{\partial b} (t,x;b,y;\mu,\nu^2)-E(t,x;b,y;\mu,\nu^2)\bigg]_{b=0}^{b=t-|x-y|}   \, \mathrm{d}y \\
& = \frac{1}{2^{\sqrt{\delta}}} \int_{x-t}^{x+t}  u_1(y) E(t,x;0,y;\mu,\nu^2)   \, \mathrm{d}y + \frac{1}{2^{\sqrt{\delta}}} \int_{x}^{x+t}  u_1(y) \bigg[b\frac{\partial E}{\partial b} (t,x;b,y;\mu,\nu^2)-E(t,x;b,y;\mu,\nu^2)\bigg]_{b=t-y+x}   \, \mathrm{d}y \\ 
& \quad + \frac{1}{2^{\sqrt{\delta}}} \int_{x-t}^{x}  u_1(y) \bigg[b\frac{\partial E}{\partial b} (t,x;b,y;\mu,\nu^2)-E(t,x;b,y;\mu,\nu^2)\bigg]_{b=t-x+y}   \, \mathrm{d}y \\
& = \frac{1}{2^{\sqrt{\delta}}} \int_{x-t}^{x+t}  u_1(y) E(t,x;0,y;\mu,\nu^2)   \, \mathrm{d}y +\frac{1}{2^{\sqrt{\delta}}} \int_{0}^{t}  u_1(x+t-b) \bigg[b\frac{\partial E}{\partial b} (t,x;b,y;\mu,\nu^2)-E(t,x;b,y;\mu,\nu^2)\bigg]_{y=x+t-b}   \, \mathrm{d}b \\ 
& \quad + \frac{1}{2^{\sqrt{\delta}}} \int_{0}^{t}  u_1(x-t+b) \bigg[b\frac{\partial E}{\partial b} (t,x;b,y;\mu,\nu^2)-E(t,x;b,y;\mu,\nu^2)\bigg]_{y=x-t+b}   \, \mathrm{d}b
\end{align*} and
\begin{align*}
\widetilde{J}_3 &= - \frac{1}{2^{\sqrt{\delta}}} \int_0^t  \int_{x-t+b}^{x+t-b} u_1(y) \frac{\mu b}{1+b}\frac{\partial E}{\partial b}(t,x;b,y;\mu,\nu^2)  \, \mathrm{d}y\, \mathrm{d}b = - \frac{1}{2^{\sqrt{\delta}}} \int_{x-t}^{x+t}  u_1(y) \int_{0}^{t-|x-y|}  \frac{\mu b}{1+b}\frac{\partial E}{\partial b}(t,x;b,y;\mu,\nu^2) \, \mathrm{d}b  \, \mathrm{d}y \\
&= -\frac{1}{2^{\sqrt{\delta}}} \int_{x-t}^{x+t}  u_1(y)\bigg[\frac{\mu b}{1+b}E(t,x;b,y;\mu,\nu^2)\bigg]_{b=0}^{b=t-|x-y|}   \mathrm{d}y + \frac{1}{2^{\sqrt{\delta}}} \int_{x-t}^{x+t}  u_1(y) \int_{0}^{t-|x-y|}  \frac{\partial  }{\partial b}\bigg(\frac{\mu b}{1+b}\bigg) E(t,x;b,y;\mu,\nu^2) \, \mathrm{d}b  \, \mathrm{d}y \\
&= -\frac{1}{2^{\sqrt{\delta}}} \int_{x-t}^{x+t}  u_1(y)\bigg[\frac{\mu b}{1+b}E(t,x;b,y;\mu,\nu^2)\bigg]_{b=t-|x-y|}   \mathrm{d}y \\
& \quad  + \frac{1}{2^{\sqrt{\delta}}} \int_{x-t}^{x+t}  u_1(y) \int_{0}^{t-|x-y|}  \bigg(-\frac{\mu b}{(1+b)^2}+\frac{\mu}{1+b}\bigg) E(t,x;b,y;\mu,\nu^2) \, \mathrm{d}b  \, \mathrm{d}y \\
&= -\frac{1}{2^{\sqrt{\delta}}} \int_{x}^{x+t}  u_1(y)\bigg[\frac{\mu b}{1+b}E(t,x;b,y;\mu,\nu^2)\bigg]_{b=t-y+x}   \mathrm{d}y-\frac{1}{2^{\sqrt{\delta}}} \int_{x-t}^{x}  u_1(y)\bigg[\frac{\mu b}{1+b}E(t,x;b,y;\mu,\nu^2)\bigg]_{b=t-x+y}   \mathrm{d}y \\
& \quad  - \frac{1}{2^{\sqrt{\delta}}} \int_0^t \int_{x-t+b}^{x+t-b}  u_1(y)  \frac{\mu b}{(1+b)^2} E(t,x;b,y;\mu,\nu^2)  \, \mathrm{d}y \, \mathrm{d}b -I_4 \\
&= -\frac{1}{2^{\sqrt{\delta}}} \int_{0}^{t}  u_1(x+t-b)\bigg[\frac{\mu b}{1+b}E(t,x;b,y;\mu,\nu^2)\bigg]_{y=x+t-b}   \mathrm{d}y-\frac{1}{2^{\sqrt{\delta}}} \int_{0}^{t}  u_1(x-t+b)\bigg[\frac{\mu b}{1+b}E(t,x;b,y;\mu,\nu^2)\bigg]_{y=x-t+b}   \mathrm{d}y \\
& \quad  - \frac{1}{2^{\sqrt{\delta}}} \int_0^t \int_{x-t+b}^{x+t-b}  u_1(y)  \frac{\mu b}{(1+b)^2} E(t,x;b,y;\mu,\nu^2)  \, \mathrm{d}y \, \mathrm{d}b -I_4.
\end{align*} Let us consider $\widetilde{J}_1$. Since
\begin{align*}
& \int_0^t  \bigg[u'_1(y) \, b E(t,x;b,y;\mu,\nu^2)\bigg]^{y=x+t-b}_{y=x-t+b} \mathrm{d}b \\
& \quad = - \int_0^t  \bigg[\frac{\partial}{\partial b} \big( u_1(x+t-b)\big)\, b E(t,x;b,x+t-b;\mu,\nu^2)+\frac{\partial}{\partial b} \big( u_1(x-t+b)\big)\, b E(t,x;b,x-t+b;\mu,\nu^2)\bigg] \mathrm{d}b \\
& \quad = -  \bigg[ u_1(x+t-b)\, b E(t,x;b,x+t-b;\mu,\nu^2)+ u_1(x-t+b) \, b E(t,x;b,x-t+b;\mu,\nu^2)\bigg]_{b=0}^{b=t}   \\
& \qquad + \int_0^t  \bigg[ u_1(x+t-b) \frac{\partial }{\partial b} \Big(b E(t,x;b,x+t-b;\mu,\nu^2)\Big) + u_1(x-t+b) \frac{\partial }{\partial b} \Big( b E(t,x;b,x-t+b;\mu,\nu^2)\Big)\bigg] \mathrm{d}b \\
& \quad = -  2^{\sqrt{\delta}} \, t  u_1(x) + \int_0^t  \bigg[ u_1(y) \bigg(E(t,x;b,y;\mu,\nu^2)+b\frac{\partial E}{\partial b} (t,x;b,y;\mu,\nu^2)-b\frac{\partial E}{\partial y} (t,x;b,y;\mu,\nu^2)\bigg)\bigg]_{y=x+t-b} \mathrm{d}b  \\
& \qquad+ \int_0^t  \bigg[ u_1(y) \bigg(E(t,x;b,y;\mu,\nu^2)+b\frac{\partial E}{\partial b} (t,x;b,y;\mu,\nu^2)+b\frac{\partial E}{\partial y} (t,x;b,y;\mu,\nu^2)\bigg)\bigg]_{y=x-t+b} \mathrm{d}b,
\end{align*} where in the last step we used \eqref{value E(t,x:t,x)}, then,
\begin{align*}
\widetilde{J}_1 &= \frac{1}{2^{\sqrt{\delta}}} \int_0^t  \bigg[ u_1'(y)\, b E(t,x;b,y;\mu,\nu^2)- u_1(y) \,b \frac{\partial E}{\partial y}(t,x;b,y;\mu,\nu^2)\bigg]^{y=x+t-b}_{y=x-t+b} \mathrm{d}b\\
 & \quad = -   t  u_1(x) +  \frac{1}{2^{\sqrt{\delta}}}  \int_0^t  \bigg[ u_1(y) \bigg(E(t,x;b,y;\mu,\nu^2)+b\frac{\partial E}{\partial b} (t,x;b,y;\mu,\nu^2)-2b\frac{\partial E}{\partial y} (t,x;b,y;\mu,\nu^2)\bigg)\bigg]_{y=x+t-b} \mathrm{d}b  \\
& \qquad+  \frac{1}{2^{\sqrt{\delta}}}  \int_0^t  \bigg[ u_1(y) \bigg(E(t,x;b,y;\mu,\nu^2)+b\frac{\partial E}{\partial b} (t,x;b,y;\mu,\nu^2)+2b\frac{\partial E}{\partial y} (t,x;b,y;\mu,\nu^2)\bigg)\bigg]_{y=x-t+b} \mathrm{d}b.
\end{align*} Summarizing,
\begin{align*}
I_2 & = \widetilde{J}_1+\widetilde{J}_2+\widetilde{J}_3+\widetilde{J}_4 \\
& =   \frac{1}{2^{\sqrt{\delta}}} \int_{x-t}^{x+t}  u_1(y) E(t,x;0,y;\mu,\nu^2)   \, \mathrm{d}y -   t  u_1(x) -I_4-I_5 \\
& \qquad+   \frac{1}{2^{\sqrt{\delta}}}  \int_0^t  \bigg[ u_1(y) \bigg(2b\frac{\partial E}{\partial b} (t,x;b,y;\mu,\nu^2)-2b\frac{\partial E}{\partial y} (t,x;b,y;\mu,\nu^2)-\frac{\mu b}{1+b}E(t,x;b,y;\mu,\nu^2)\bigg)\bigg]_{y=x+t-b} \mathrm{d}b  \\
& \qquad+  \frac{1}{2^{\sqrt{\delta}}}  \int_0^t  \bigg[ u_1(y) \bigg(2b\frac{\partial E}{\partial b} (t,x;b,y;\mu,\nu^2)+2b\frac{\partial E}{\partial y} (t,x;b,y;\mu,\nu^2)-\frac{\mu b}{1+b}E(t,x;b,y;\mu,\nu^2)\bigg)\bigg]_{y=x-t+b} \mathrm{d}b.
\end{align*} Using the identity \eqref{relation 2 dE/db -2 dE/dy}, from the last equality we get
\begin{align}
I_2+I_4+I_5=  \frac{1}{2^{\sqrt{\delta}}} \int_{x-t}^{x+t}  u_1(y) E(t,x;0,y;\mu,\nu^2)   \, \mathrm{d}y -   t  u_1(x) . \label{I2+I4+I5}
\end{align} Finally, if we combine \eqref{I1+I3} and \eqref{I2+I4+I5}, we arrive at
\begin{align*}
w(t,x)& = I_1+I_2+I_3+I_4+I_5 \\ 
& =  -  u_0(x) - t  u_1(x)+ \frac{1}{2}(1+t)^{-\frac{\mu}{2}}\big[ u_0(x+t)+ u_0(x-t) \big]+ \frac{1}{2^{\sqrt{\delta}}} \int_{x-t}^{x+t}  u_0(y) \bigg[-\frac{\partial E}{\partial b} (t,x;b,y;\mu,\nu^2)\bigg]_{b=0} \mathrm{d}y   \\
 & \quad +\frac{1}{2^{\sqrt{\delta}}} \int_{x-t}^{x+t}  \big( u_1(y)+\mu\, u_0(y)\big) E(t,x;0,y;\mu,\nu^2)   \, \mathrm{d}y .
\end{align*} So, we proved that $u^{\h}$ defined in \eqref{homogeneous sol 1d} solves \eqref{homog CP 1d}. 

\subsection{Final remarks on the 1d case}
Combining the results from Subsections \ref{Subsection inhom 1d CP vanishing data} and \ref{Subsection hom 1d CP}, we see that $u=u^{\h}+u^{\ih}$ is a solution of the Cauchy problem \eqref{inhomog CP} for $n=1$ as stated in Theorem \ref{Thm representation formula 1d case}. 

Let us remark that for $\mu=\nu^2=0$ we have $\delta=1$ and $\mathsf{F}(0,0;1;z)=1$, so that $E(t,x;b,y;0,0)=1$. This means that \eqref{representation formula 1d case} coincides with d'Alembert's representation formula for $\mu=\nu^2=0$. As in d'Alembert's representation formula for the classical wave equation, in the one dimensional case we have no loss of regularity for the solution in comparison with initial data. However, differently from d'Alembert's representation formula, we have that the first data appears, in general, also in an integral term.

\section{Odd dimensional case} \label{Section n odd}

In this section we will prove Theorem \ref{Thm representation formula n dimensional odd case} with the method of spherical means (see \cite{John55} for further details).

\subsection{Spherical means}

Let $u=u(t,x)$ solve \eqref{inhomog CP}. We define
\begin{align*}
I_r[u](t,x) \doteq \frac{1}{\omega_{n-1}}\int_{|\omega|=1}u(t,x+r \omega) \, \mathrm{d}\sigma_\omega = \frac{1}{\omega_{n-1}r^{n-1}}\int_{\partial B_r(x)}u(t,z) \, \mathrm{d}\sigma_z = \fint_{\partial B_r(x)}u(t,z) \, \mathrm{d}\sigma_z,
\end{align*} where $\omega_{n-1}$ is the $(n-1)$-dimensional measure of the unit sphere of $\mathbb{R}^n$ and, analogously, 
\begin{align*}
I_r[u_j](x) \doteq  \fint_{\partial B_r(x)}u_j(z) \, \mathrm{d}\sigma_z \qquad \mbox{for} \ \ j=0,1.
\end{align*} Moreover, we introduce the operator $\Omega_r$ as follows:
\begin{align}
\Omega_r[u](t,x) &\doteq \bigg(\frac{1}{r}\frac{\partial}{\partial r}\bigg)^{k-1}\Big(r^{2k-1} I_r[u](t,x)\Big), \label{def Omega u}\\
\Omega_r[u_j](x) &\doteq \bigg(\frac{1}{r}\frac{\partial}{\partial r}\bigg)^{k-1}\Big(r^{2k-1} I_r[u_j](x)\Big) \qquad \mbox{for} \ \ j=0,1, \notag
\end{align} where $k$ satisfies the relation $n=2k+1$. We remark that the equality $$I_r[u_{tt}](t,x)=\fint_{\partial B_r(x)}u_{tt}(t,z) \, \mathrm{d}\sigma_z= \left(\frac{\partial}{\partial t}\right)^2 \fint_{\partial B_r(x)}u(t,z) \, \mathrm{d}\sigma_z=\left(\frac{\partial}{\partial t}\right)^2 \Big(I_r[u](t,x)\Big)$$ implies
\begin{align*}
\Omega_r[u_{tt}](t,x) &= \bigg(\frac{1}{r}\frac{\partial}{\partial r}\bigg)^{k-1}\Big(r^{2k-1} I_r[u_{tt}](t,x)\Big) = \bigg(\frac{1}{r}\frac{\partial}{\partial r}\bigg)^{k-1}\left(r^{2k-1} \left(\frac{\partial}{\partial t}\right)^2 \Big(I_r[u](t,x)\Big)\right) \\  &= \left(\frac{\partial}{\partial t}\right)^2 \bigg(\frac{1}{r}\frac{\partial}{\partial r}\bigg)^{k-1}\left(r^{2k-1}  \Big(I_r[u](t,x)\Big)\right) = \left(\frac{\partial}{\partial t}\right)^2 \Omega_r[u](t,x).
\end{align*} Similarly, one can prove
\begin{align*}
\Omega_r\left[\frac{\mu}{1+t}u_{t}\right](t,x)& =\frac{\mu}{1+t} \bigg(\frac{\partial}{\partial t}\bigg)\, \Omega_r[u](t,x)  , \\
\Omega_r\left[\frac{\nu^2}{(1+t)^2}u \right](t,x)& =\frac{\nu^2}{(1+t)^2} \,  \Omega_r[u](t,x) . 
\end{align*} Due to the linearity of the operator $\Omega_r$, we get that $\Omega_r[u]$ solves
\begin{align*}
\left(\frac{\partial}{\partial t}\right)^2 \Omega_r[u](t,x)+ \frac{\mu}{1+t} \left(\frac{\partial}{\partial t}\right)\, \Omega_r[u](t,x)+ \frac{\nu^2}{(1+t)^2} \,  \Omega_r[u](t,x) = \Omega_r[\Delta u](t,x)+\Omega_r[f](t,x). 
\end{align*}	 Next, we shall express in a more convenient way the action of $\Omega_r$ on the Laplacian of $u$. This relation is well-known in the literature, but we will prove it in few steps for the ease of the reader. Let us calculate the derivative of order 2 of $I_r[u]$. By Green's formula we get
\begin{align*}
\frac{\partial}{\partial r} I_r[u](t,x) & = \frac{1}{\omega_{n-1}}\int_{|\omega|=1}\nabla u(t,x+r \omega) \cdot \frac{\partial}{\partial r}\big(x+r \omega \big)\, \mathrm{d}\sigma_\omega =  \frac{1}{\omega_{n-1}}\int_{|\omega|=1}\nabla u(t,x+r \omega) \cdot \omega \, \mathrm{d}\sigma_\omega \\
& =  \frac{r}{\omega_{n-1}}\int_{|\omega|\leqslant1}\Delta u(t,x+r \omega) \, \mathrm{d}\omega =  \frac{1}{\omega_{n-1} r^{n-1}}\int_{B_r(x)}\Delta u(t,z) \, \mathrm{d}z \\ & =  \frac{1}{\omega_{n-1} r^{n-1}}\int_0^r\int_{\partial B_\varrho(x)}\Delta u(t,\omega) \, \mathrm{d}\sigma_{\omega} \, \mathrm{d}\varrho.
\end{align*} A further differentiation with respect to $r$ provides
\begin{align*}
\bigg(\frac{\partial}{\partial r}\bigg)^2 I_r[u](t,x) &= \frac{1}{\omega_{n-1} r^{n-1}}\int_{\partial B_r(x)}\Delta u(t,\omega) \, \mathrm{d}\sigma_{\omega}   -\frac{n-1}{\omega_{n-1} r^{n}}\int_0^r\int_{\partial B_\varrho(x)}\Delta u(t,\omega) \, \mathrm{d}\sigma_{\omega} \, \mathrm{d}\varrho \\
&= \frac{1}{\omega_{n-1} r^{n-1}}\int_{\partial B_r(x)}\Delta u(t,\omega) \, \mathrm{d}\sigma_{\omega}   -\frac{n-1}{r} \frac{\partial}{\partial r} I_r[u](t,x) ,
\end{align*} that is,
\begin{align*}
\bigg(\frac{\partial}{\partial r}\bigg)^2 I_r[u](t,x)+\frac{n-1}{r} \frac{\partial}{\partial r} I_r[u](t,x)=  \fint_{B_r(x)}\Delta u(t,\omega) \, \mathrm{d}\sigma_{\omega} =  I_r[\Delta u](t,x).
\end{align*} The previous relation implies
\begin{align*}
\bigg(\frac{\partial}{\partial r}\bigg)^2 \Omega_r[u](t,x) & =\bigg(\frac{\partial}{\partial r}\bigg)^2 \bigg(\frac{1}{r}\frac{\partial }{\partial r}\bigg)^{k-1}\left(r^{2k-1} I_r[u](t,x)\right) = \bigg(\frac{1}{r}\frac{\partial}{\partial r}\bigg)^{k}\left(r^{2k} \frac{\partial}{\partial r}I_r[u](t,x)\right) \\
&= \bigg(\frac{1}{r}\frac{\partial}{\partial r}\bigg)^{k-1}\left[r^{2k-1} \bigg(\frac{\partial}{\partial r}\bigg)^2 I_r[u](t,x)+2k \, r^{2k-2} \frac{\partial}{\partial r} I_r[u](t,x)\right] \\
&= \bigg(\frac{1}{r}\frac{\partial}{\partial r}\bigg)^{k-1}\left[r^{2k-1}\left( \bigg(\frac{\partial}{\partial r}\bigg)^2 I_r[u](t,x)+\frac{n-1}{r}\frac{\partial}{\partial r} I_r[u](t,x)\right)\right] = \bigg(\frac{1}{r}\frac{\partial}{\partial r}\bigg)^{k-1}\left[r^{2k-1}\left(  I_r[\Delta u](t,x)\right)\right]\\ &=  \Omega_r[\Delta u](t,x),
\end{align*}
where in the second equality we used the identity 
\begin{align*}
\bigg(\frac{\mathrm{d}}{\mathrm{d}r}\bigg)^2 \bigg(\frac{1}{r}\frac{\mathrm{d}}{\mathrm{d} r}\bigg)^{k-1}\left(r^{2k-1} \phi(r)\right) = \bigg(\frac{1}{r}\frac{\mathrm{d}}{\mathrm{d} r}\bigg)^{k}\left(r^{2k} \frac{\mathrm{d} \phi}{\mathrm{d}r}(r)\right)
\end{align*} whose validity can be proved by using an inductive argument (cf. \cite[Lemma 2, Section 2.4.1]{Evans}). If we introduce the function $v=v(r,t;x)=\Omega_r[u](t,x)$, then, $v$ solves the following initial boundary value problem depending on the parameter $x\in \mathbb{R}^n$:
\begin{align}\label{IBV problem}
\begin{cases}
\partial_t^2 v(r,t;x)-\partial_r^2 v(r,t;x)+\frac{\mu}{1+t}\partial_t v(r,t;x)+\frac{\nu^2}{(1+t)^2} v(r,t;x) = \Omega_r[f](t,x), \qquad t>0, r>0, \\
v(r,0;x)= \Omega_r[u_0](x), \ \ \, \qquad r>0, \\
\partial_t v(r,0;x)= \Omega_r[u_1](x), \qquad r>0, \\
v(0,t;x)=0, \qquad \qquad \qquad \ \, t \geqslant 0.
\end{cases}
\end{align} In order to get the boundary condition in \eqref{IBV problem}, we employed the following formula
\begin{align}\label{action of the Omega operator}
 \bigg(\frac{1}{r}\frac{\mathrm{d}}{\mathrm{d} r}\bigg)^{k-1}\left(r^{2k-1} \phi(r)\right) &= \sum_{j=0}^{k-1} \beta^{(k)}_j r^{j+1}\frac{\mathrm{d}^j \phi}{\mathrm{d}r^j}(r),
\end{align} where the constants $\{\beta^{(k)}_j\}_{ j=0,\cdots, k-1}$ are independent of $\phi$ and, in particular, $\beta^{(k)}_0=(2k-1)!!$ (see also, for example, \cite[Lemma 2, Section 2.4.1]{Evans}).

Since $I_r[u](t,x)$ can be extended to an even function for $r<0$, $\Omega_r[u](t,x)$ has a natural extension as odd function with respect to $r$ for $r<0$, due to \eqref{def Omega u}. We denote the odd extensions of $v$, $\Omega_r[f]$ and $\Omega_r[u_j]$ for $j=0,1$ by
\begin{align*}
\widetilde{v}(r,t;x)& \doteq \begin{cases} v(r,t;x) & \mbox{if} \ \ r\geqslant 0, \\ -v(-r,t;x) & \mbox{if} \ \ r\leqslant 0 , \end{cases} \qquad  \widetilde{\Omega}_r[f](t,x) \doteq\begin{cases} \Omega_r[f](t,x) & \mbox{if} \ \ r\geqslant 0, \\ -\Omega_{-r}[f](t,x) & \mbox{if} \ \ r\leqslant 0 , \end{cases} \\   \widetilde{\Omega}_r[u_j](x) & \doteq\begin{cases} \Omega_r[u_j](x) & \mbox{if} \ \ r\geqslant 0, \\ -\Omega_{-r}[u_j](x) & \mbox{if} \ \ r\leqslant 0 , \end{cases} 
\end{align*} respectively. Therefore, $\widetilde{v}$ solves the Cauchy problem depending on the parameter $x\in \mathbb{R}^n$
\begin{align}\label{Cauchy problem for the spherical mean}
\begin{cases}
\partial_t^2 \widetilde{v}(r,t;x)-\partial_r^2 \widetilde{v}(r,t;x)+\frac{\mu}{1+t}\partial_t \widetilde{v}(r,t;x)+\frac{\nu^2}{(1+t)^2} \widetilde{v}(r,t;x) = \widetilde{\Omega}_r[f](t,x), \qquad t>0, r\in \mathbb{R}, \\
\widetilde{v}(r,0;x)=\widetilde{\Omega}_r[u_0](x), \ \ \, \qquad  r\in \mathbb{R}, \\
\partial_t \widetilde{v}(r,0;x)= \widetilde{\Omega}_r[u_1](x), \qquad  r\in \mathbb{R}.
\end{cases}
\end{align} Hence, \eqref{Cauchy problem for the spherical mean} is a Cauchy problem for an inhomogeneous linear wave equation with scale-invariant damping and mass in the one dimensional case. Thanks to Theorem \ref{Thm representation formula 1d case}, we have an explicit representation formula for $\widetilde{v}$, namely,
\begin{align}
\widetilde{v}(r,t;x)& =\frac{1}{2}(1+t)^{-\frac{\mu}{2}}\Big(\widetilde{\Omega}_{r+t}[u_0](x)+\widetilde{\Omega}_{r-t}[u_0](x)\Big)+\frac{1}{2^{\sqrt{\delta}}}\int_{r-t}^{r+t} \widetilde{\Omega}_s[u_0](x) K_0(t,r;s;\mu,\nu^2)\, \mathrm{d}s \notag \notag \\ & \quad +\frac{1}{2^{\sqrt{\delta}}}\int_{r-t}^{r+t}\Big(\widetilde{\Omega}_s[u_1](x)+\mu \, \widetilde{\Omega}_s[u_0](x)\Big) K_1(t,r;s;\mu,\nu^2)\, \mathrm{d}s  +\frac{1}{2^{\sqrt{\delta}}} \int_0^t \int_{r-t+b}^{r+t-b} \widetilde{\Omega}_s[f](b,x) E(t,r;b,s;\mu,\nu^2) \, \mathrm{d}s\, \mathrm{d}b. \label{representation formula v tilde}
\end{align}
In the next subsection, we will apply a limit argument to \eqref{representation formula v tilde} in order to derive a representation formula for \eqref{inhomog CP} in the odd dimensional case.

\subsection{Representation formula via a limit argument} From \eqref{action of the Omega operator} it follows that
\begin{align*}
u(t,x)= \lim_{r\to 0} I_r[u](t,x) =  \lim_{r\to 0} \frac{1}{\beta^{(k)}_0 r} \Omega_r[u](t,x) =  \frac{1}{(n-2)!!}\lim_{r\to 0} \frac{\widetilde{v}(r,t;x)}{ r} .
\end{align*} Our strategy consists in using \eqref{representation formula v tilde} in order to calculate the previous limit. We will consider separately the four addends that appear in \eqref{representation formula v tilde}. Fixed $t>0$, since we  will calculate the limit as $r\to 0$ we may assume without loss of generality that $r<t$, thus,
\begin{align*}
\frac{1}{r} \Big(\widetilde{\Omega}_{r+t}[u_0](x)+\widetilde{\Omega}_{r-t}[u_0](x)\Big) & = \frac{1}{r} \Big(\Omega_{r+t}[u_0](x)-\Omega_{t-r}[u_0](x)\Big) \xrightarrow [r\to 0] {} 2\frac{\partial}{\partial t} \, \Omega_{t}[u_0](x).
\end{align*}
 For the integral containing the kernel function $K_0$, we have
 \begin{align*}
\frac{1}{r} \int_{r-t}^{r+t} \widetilde{\Omega}_s[u_0](x) K_0(t,r;s;\mu,\nu^2)\, \mathrm{d}s & = \frac{1}{r} \int_{-t}^{t} \widetilde{\Omega}_{s+r}[u_0](x) K_0(t,r;s+r;\mu,\nu^2)\, \mathrm{d}s \\ &  = \frac{1}{r} \int_{0}^{t} \Big[ \widetilde{\Omega}_{s+r}[u_0](x) K_0(t,r;s+r;\mu,\nu^2)+\widetilde{\Omega}_{r-s}[u_0](x) K_0(t,r;-s+r;\mu,\nu^2)\Big]\, \mathrm{d}s  \\ &  =  \int_{0}^{t} \frac{1}{r} \Big[ \widetilde{\Omega}_{s+r}[u_0](x) +\widetilde{\Omega}_{r-s}[u_0](x)\Big] K_0(t,r;s+r;\mu,\nu^2)\, \mathrm{d}s,
 \end{align*} where in the last step we used that $K_0(t,r;s+r;\mu,\nu^2)$ is even with respect to $s$; this follows immediately from the fact that $E(t,r;b,s+r;\mu,\nu^2)$ is even with respect to $s$ and from the definition \eqref{def K0(t,x;y)}. Letting $r\to 0$ in the last expression we have
 \begin{align*}
& \frac{1}{r} \int_{r-t}^{r+t} \widetilde{\Omega}_s[u_0](x) K_0(t,r;s;\mu,\nu^2)\, \mathrm{d}s\\ \qquad  & =  \int_{0}^{t} \frac{1}{r} \Big[ \Omega_{s+r}[u_0](x) -\Omega_{s-r}[u_0](x)\Big] K_0(t,r;s+r;\mu,\nu^2)\, \mathrm{d}s \xrightarrow [r\to 0] {} 2 \int_{0}^{t}  \frac{\partial}{\partial s} \, \Omega_{s}[u_0](x) K_0(t,0;s;\mu,\nu^2)\, \mathrm{d}s.
 \end{align*} In an analogous way, $K_1(t,r;s+r;\mu,\nu^2)$ being an even function with respect to $s$, we get
 \begin{align*}
& \frac{1}{r} \int_{r-t}^{r+t} \widetilde{\Omega}_s[u_1+\mu u_0](x)K_1(t,r;s;\mu,\nu^2)\, \mathrm{d}s   
\xrightarrow [r\to 0] {} 2 \int_{0}^{t}  \frac{\partial}{\partial s} \, \Omega_{s}[u_1+\mu u_0](x) K_1(t,0;s;\mu,\nu^2)\, \mathrm{d}s.
 \end{align*} Finally, we consider the integral term involving the source term. It results
 \begin{align*}
 \frac{1}{r}  \int_0^t &\int_{r-t+b}^{r+t-b} \widetilde{\Omega}_s[f](b,x) E(t,r;b,s;\mu,\nu^2) \, \mathrm{d}s\, \mathrm{d}b = \frac{1}{r} \int_0^t \int_{-t+b}^{t-b} \widetilde{\Omega}_{s+r}[f](b,x) E(t,r;b,s+r;\mu,\nu^2) \, \mathrm{d}s\, \mathrm{d}b \\ 
 &= \frac{1}{r} \int_0^t \int_{0}^{t-b}\Big[ \widetilde{\Omega}_{s+r}[f](b,x) E(t,r;b,s+r;\mu,\nu^2)+\widetilde{\Omega}_{r-s}[f](b,x) E(t,r;b,-s+r;\mu,\nu^2)\Big] \, \mathrm{d}s\, \mathrm{d}b \\
 & = \frac{1}{r} \int_0^t \int_{0}^{t-b}\Big[ \widetilde{\Omega}_{s+r}[f](b,x) +\widetilde{\Omega}_{r-s}[f](b,x) \Big] E(t,r;b,s+r;\mu,\nu^2) \,\mathrm{d}s\, \mathrm{d}b,
 \end{align*} where in the last step we used the property $E(t,r;b,s+r;\mu,\nu^2)=E(t,r;b,-s+r;\mu,\nu^2)$. Consequently, letting $r\to 0$, we have
 \begin{align*}
  & \frac{1}{r}  \int_0^t \int_{r-t+b}^{r+t-b} \widetilde{\Omega}_s[f](b,x) E(t,r;b,s;\mu,\nu^2) \, \mathrm{d}s\, \mathrm{d}b  \\ 
  & =\int_0^t \int_{0}^{t-b}\!  \frac{1}{r} \Big[ \Omega_{s+r}[f](b,x) -\Omega_{s-r}[f](b,x) \Big] E(t,r;b,s+r;\mu,\nu^2) \, \mathrm{d}s\, \mathrm{d}b \xrightarrow [r\to 0] {} 2 \! \int_{0}^{t}  \! \int_{0}^{t-b}   \! \frac{\partial}{\partial s} \, \Omega_{s}[f](b,x) E(t,0;b,s;\mu,\nu^2) \, \mathrm{d}s\, \mathrm{d}b.
 \end{align*}
 Summarizing, we proved
 \begin{align}
 (n-2)!! \,  u(t,x)  = \lim_{r\to 0} \frac{\widetilde{v}(r,t;x)}{ r}  & = (1+t)^{-\frac{\mu}{2}} \frac{\partial}{\partial t} \, \Omega_{t}[u_0](x)+\frac{1}{2^{\sqrt{\delta}-1}}\int_{0}^{t}  \frac{\partial}{\partial s} \, \Omega_{s}[u_0](x) K_0(t,0;s;\mu,\nu^2)\, \mathrm{d}s \notag \\ 
& \quad +\frac{1}{2^{\sqrt{\delta}-1}}\int_{0}^{t}  \frac{\partial}{\partial s} \, \Omega_{s}[u_1+\mu u_0](x) K_1(t,0;s;\mu,\nu^2)\, \mathrm{d}s \notag \\ & \quad +\frac{1}{2^{\sqrt{\delta}-1}} \int_{0}^{t}  \int_{0}^{t-b}    \frac{\partial}{\partial s} \, \Omega_{s}[f](b,x) E(t,0;b,s;\mu,\nu^2) \, \mathrm{d}s\, \mathrm{d}b. \label{final eq case odd}
 \end{align} According to what we recall in the introduction, more precisely the representation given in \eqref{wave sol operator n odd}, we have
 \begin{align*}
 w[\varphi](t,x)= \frac{1}{(n-2)!!} \frac{\partial}{\partial t} \, \Omega_t[\varphi](x).
 \end{align*} So, \eqref{final eq case odd} implies easily \eqref{representation formula n dimensional odd case}.

\section{Even dimensional case: method of descent}  \label{Section n even}

In this section we prove Theorem \ref{Thm representation formula n dimensional even case}, by using the so-called \emph{method of descent}. Let us consider $u=u(t,x)$ solution of \eqref{inhomog CP} when $n\geqslant 2$ is an even integer. Then, we can consider formally $u$ as a function defined on $[0,\infty)\times \mathbb{R}^{n+1}$, by setting
\begin{align*}
\bar{u}(t,x,x_{n+1})\doteq  u(t,x) \qquad \mbox{for any} \ \ t\geqslant 0, (x,x_{n+1})\in \mathbb{R}^{n+1}.
\end{align*} Then, $\bar{u}$ solves
\begin{align}\label{Cauchy Problem descent}
\begin{cases} \bar{u}_{tt}-\sum_{j=1}^{n+1} \bar{u}_{x_jx_j} +\frac{\mu}{1+t}\bar{u}_t+\frac{\nu^2}{(1+t)^2}\bar{u}=\bar{f}(t,x,x_{n+1}), &  (x,x_{n+1})\in \mathbb{R}^{n+1}, \ t>0,\\
\bar{u}(0,x,x_{n+1})=\bar{u}_0(x,x_{n+1}), & (x,x_{n+1})\in \mathbb{R}^{n+1}, \\ \bar{u}_t(0,x,x_{n+1})=\bar{u}_1(x,x_{n+1}), & (x,x_{n+1})\in \mathbb{R}^{n+1},
\end{cases}
\end{align} where
\begin{align*}
\bar{u}_0(x,x_{n+1}) & \doteq u_0(x), \quad \bar{u}_1(x,x_{n+1}) \doteq  u_1(x) \qquad \mbox{for any} \ \  (x,x_{n+1})\in \mathbb{R}^{n+1}, \\
\bar{f}(t,x,x_{n+1})& \doteq  f(t,x)  \ \ \quad \qquad \qquad \qquad \qquad \qquad \mbox{for any} \ \ t\geqslant 0, (x,x_{n+1})\in \mathbb{R}^{n+1}.
\end{align*} Due to the fact that $n+1$ is an odd integer, we can use Theorem \ref{Thm representation formula n dimensional odd case} to get a representation formula for $\bar{u}$. Let us underline that $\bar{u}$ depends only formally on $x_{n+1}$, so we can consider without loss of generality the restriction of $\bar{u}$ on the hyperplane $\{x_{n+1}=0\}$. For the sake of readability we will denote by $\bar{B}_r(z)$ the ball around $z$ with radius $r$ in $\mathbb{R}^{n+1}$ and we will keep the usual notation for balls in $\mathbb{R}^n$. According to \eqref{representation formula n dimensional odd case} and \eqref{wave sol operator n odd}, we have
\begin{align}
u(t,x) & = \bar{u}(t,x,0)\notag  \\ &= \frac{1}{(n-1)!!} (1+t)^{-\frac{\mu}{2}}\bigg(\frac{\partial}{\partial t}\bigg) \bigg(\frac{1}{t}\frac{\partial}{\partial t}\bigg)^{\frac{n}{2}-1} \left( t^{n-1} \fint_{\partial \bar{B}_t(x,0)} \bar{u}_0(z,z_{n+1})\, \mathrm{d}\sigma_{(z,z_{n+1})} \right)\notag \\
& \quad +\frac{2^{1-\sqrt{\delta}}}{(n-1)!!}\int_0^t \bigg(\frac{\partial}{\partial s}\bigg) \bigg(\frac{1}{s}\frac{\partial}{\partial s}\bigg)^{\frac{n}{2}-1} \left( s^{n-1} \fint_{\partial \bar{B}_s(x,0)} \bar{u}_0(z,z_{n+1}) \, \mathrm{d}\sigma_{(z,z_{n+1})} \right) K_0(t,0;s; \mu,\nu^2) \, \mathrm{d}s \notag \\
& \quad +\frac{2^{1-\sqrt{\delta}}}{(n-1)!!}\int_0^t \bigg(\frac{\partial}{\partial s}\bigg) \bigg(\frac{1}{s}\frac{\partial}{\partial s}\bigg)^{\frac{n}{2}-1} \left( s^{n-1} \fint_{\partial \bar{B}_s(x,0)} \big(\bar{u}_1(z,z_{n+1})+\mu \bar{u}_0(z,z_{n+1})\big)\, \mathrm{d}\sigma_{(z,z_{n+1})} \right) K_1(t,0;s; \mu,\nu^2) \, \mathrm{d}s \notag \\
& \quad +\frac{2^{1-\sqrt{\delta}}}{(n-1)!!}\int_0^t \int_0^{t-b} \bigg(\frac{\partial}{\partial s}\bigg) \bigg(\frac{1}{s}\frac{\partial}{\partial s}\bigg)^{\frac{n}{2}-1} \left( s^{n-1} \fint_{\partial \bar{B}_s(x,0)} \bar{f}(b,z,z_{n+1})\, \mathrm{d}\sigma_{(z,z_{n+1})} \right) E(t,0;b,s; \mu,\nu^2) \, \mathrm{d}s \, \mathrm{d}b. \label{final eq case even}
\end{align} The next step is to rewrite the surface integrals in $\mathbb{R}^{n+1}$ as domain integrals in $\mathbb{R}^n$. We remark that $$\partial\bar{B}_r(x,0)=\left\{(y,y_{n+1})\in \mathbb{R}^{n+1}: y_{n+1}=\pm (r^2-|y-x|^2)^{1/2}\right\}.$$ Therefore, $\partial\bar{B}_r(x,0)\cap \{(y,y_{n+1}): y_{n+1}\geqslant 0\}$ is the graph of the function
\begin{align*}
\gamma: y\in B_r(x)\rightarrow \gamma(y) \doteq (r^2-|y-x|^2)^{1/2}
\end{align*} and, similarly, $\partial\bar{B}_r(x,0)\cap \{(y,y_{n+1}): y_{n+1}\leqslant 0\}$ is the graph of the function $-\gamma$. Since, $$\nabla \gamma (y)=- \frac{ (y-x)}{(r^2-|y-x|^2)^{1/2}},$$ if $\varphi$ is a function defined on $\mathbb{R}^n$ and $\bar{\varphi}$ denotes its trivial extension as a function of $n+1$ variables (we have in mind the cases in which  $\varphi$ is equal to $u_0,u_1$ or $f(t,\cdot)$), then,
\begin{align*}
r^{n-1}\fint_{\partial \bar{B}_r(x,0)} \bar{\varphi} (z,z_{n+1})\, \mathrm{d}\sigma_{(z,z_{n+1})} & = \frac{1}{\omega_n r}\int_{\partial \bar{B}_r(x,0)} \bar{\varphi} (z,z_{n+1}) \, \mathrm{d}\sigma_{(z,z_{n+1})} = \frac{2}{\omega_n r}\int_{B_r(x)} \bar{\varphi} (z,\gamma(z)) \sqrt{1+|\nabla \gamma(z)|^2}\, \mathrm{d}z \\ &= \frac{2}{\omega_n }\int_{B_r(x)} \frac{\varphi (z)}{(r^2-|y-x|^2)^{1/2}}\, \mathrm{d}z = \frac{2\omega_{n-1}}{\omega_n  n}r^n\fint_{B_r(x)} \frac{\varphi (z)}{(r^2-|y-x|^2)^{1/2}}\, \mathrm{d}z,
\end{align*} where  the factor $2$ in the second step is due to the fact that $\partial \bar{B}_r(x,0)$ consists of two hemispheres. It is well-known that the measure of the $(n-1)$-dimensional unit sphere of $\mathbb{R}^n$ is $$\omega_{n-1}=\frac{2 \pi^{\frac{n}{2}}}{\Gamma(\frac{n}{2})},$$ where $\Gamma$ is the Euler integral function of the second kind. Consequently, using the recursive relation $\Gamma(z+1)=z\Gamma(z)$ iteratively and the values $\Gamma(1)=1$, $\Gamma(\frac{1}{2})=\sqrt{\pi}$, we get
\begin{align*}
\frac{2\omega_{n-1}}{\omega_n  n} =\frac{2}{\sqrt{\pi} n} \frac{\Gamma(\frac{n+1}{2})}{\Gamma(\frac{n}{2})} =\frac{2}{\sqrt{\pi} n} \frac{\frac{n-1}{2}\cdot\frac{n-3}{2}\cdots \frac{1}{2}\cdot \Gamma(\frac{1}{2})}{\frac{n-2}{2}\cdot \frac{n-4}{2}\cdots \frac{2}{2}\cdot\Gamma(1)} =\frac{2}{\sqrt{\pi} n} \frac{2^{-\frac{n}{2}}(n-1)!!\sqrt{\pi}}{2^{-\frac{n}{2}+1}(n-2)!!} = \frac{(n-1)!!}{n!!}.
\end{align*} Therefore,
\begin{align*}
r^{n-1}\fint_{\partial \bar{B}_r(x,0)} \bar{\varphi} (z,z_{n+1})\, \mathrm{d}\sigma_{(z,z_{n+1})} & = \frac{(n-1)!!}{n!!} r^n\fint_{B_r(x)} \frac{\varphi (z)}{(r^2-|y-x|^2)^{1/2}}\, \mathrm{d}z.
\end{align*} Hence, applying the previous relation to \eqref{final eq case even}, we get finally
\begin{align*}
u(t,x) = \bar{u}(t,x,0) &= \frac{1}{n!!} (1+t)^{-\frac{\mu}{2}}\bigg(\frac{\partial}{\partial t}\bigg) \bigg(\frac{1}{t}\frac{\partial}{\partial t}\bigg)^{\frac{n}{2}-1} \left(t^n\fint_{B_t(x)} \frac{u_0(z)}{(t^2-|y-x|^2)^{1/2}}\, \mathrm{d}z  \right) \\
& \quad +\frac{2^{1-\sqrt{\delta}}}{n!!}\int_0^t \bigg(\frac{\partial}{\partial s}\bigg) \bigg(\frac{1}{s}\frac{\partial}{\partial s}\bigg)^{\frac{n}{2}-1} \left( s^n\fint_{B_s(x)} \frac{u_0(z)}{(s^2-|y-x|^2)^{1/2}}\, \mathrm{d}z\right) K_0(t,0;s; \mu,\nu^2) \, \mathrm{d}s  \\
& \quad +\frac{2^{1-\sqrt{\delta}}}{n!!}\int_0^t \bigg(\frac{\partial}{\partial s}\bigg) \bigg(\frac{1}{s}\frac{\partial}{\partial s}\bigg)^{\frac{n}{2}-1} \left( s^n\fint_{B_s(x)} \frac{u_1(z)+\mu \, u_0(z)}{(s^2-|y-x|^2)^{1/2}}\, \mathrm{d}z\right) K_1(t,0;s; \mu,\nu^2) \, \mathrm{d}s  \\
& \quad +\frac{2^{1-\sqrt{\delta}}}{n!!}\int_0^t \int_0^{t-b} \bigg(\frac{\partial}{\partial s}\bigg) \bigg(\frac{1}{s}\frac{\partial}{\partial s}\bigg)^{\frac{n}{2}-1} \left( s^n\fint_{B_s(x)} \frac{f (b,z)}{(s^2-|y-x|^2)^{1/2}}\, \mathrm{d}z\right) E(t,0;b,s; \mu,\nu^2) \, \mathrm{d}s \, \mathrm{d}b.
\end{align*} So, combining \eqref{representation formula n dimensional odd case} and \eqref{wave sol operator n even}, we proved Theorem \ref{Thm representation formula n dimensional even case}.

\section{Final remarks} \label{Section final rem}

In this section, we list some straightforward consequences of Theorems \ref{Thm representation formula n dimensional odd case} and \ref{Thm representation formula n dimensional even case} and some relations/connections of the representation formulae in \eqref{representation formula 1d case} and \eqref{representation formula n dimensional odd case} with representation formulae for other hyperbolic equations with time-dependent coefficients.

\paragraph{Loss of regularity}
First, we remark that in the multidimensional case $n\geq 2$ we have a loss of regularity for the solution of \eqref{inhomog CP} in comparison with the regularity of initial data, differently from the one-dimensional case. Indeed, according to Theorem \ref{Thm representation formula n dimensional odd case} in the odd dimensional case we have a loss of regularity of order $\frac{n-1}{2}$, while in the even dimensional case the loss of regularity has order $\frac{n}{2}$, according to Theorem \ref{Thm representation formula n dimensional even case}.

\paragraph{Domain of dependence} From \eqref{representation formula 1d case} and from \eqref{representation formula n dimensional odd case} (combined with \eqref{wave sol operator n odd} and \eqref{wave sol operator n even}) we see that  the domain of dependence  in the point $(t_0,x_0)\in [0,\infty)\times \mathbb{R}^n$ for the solution of \eqref{inhomog CP}  is 
\begin{align*}
\Omega(t_0,x_0)=\Big\{(t,x)\in [0,\infty)\times \mathbb{R}^n: t\in [0,t_0], \, |x-x_0|\leqslant (t_0-t) \Big\}.
\end{align*} In other words, $u(t_0,x_0)$ depends on the value of  $f$ in $\Omega(t_0,x_0)$ and the values of $u_0,u_1$ in $\Omega(t_0,x_0)\cap\{t=0\}$. So, in the case of scale-invariant models from the representation formulae that we proved in this work we found in a different way a property that is known to be true in a more general frame for hyperbolic models (see for example \cite[Theorem 2.2 in Chapter 1]{Sog08}).

\paragraph{Finite speed of propagation of perturbations} Of course, we may change our prospective and analyze how the initial data and the source term influence the behavior of the solution. Let us assume that $u_0,u_1$ are compactly supported  in $B_R(0)$ and that $\supp f \subset K_R\doteq \{(t,x)\in[0,\infty)\times \mathbb{R}^n: |x|\leq R+t\}$. Then, the solution itself has support contained in the forward conical domain $K_R$. This follows immediately by \eqref{representation formula 1d case} and \eqref{representation formula n dimensional odd case}. Indeed, in order to get not identically vanishing integrands in \eqref{representation formula 1d case} and \eqref{representation formula n dimensional odd case} or an actual influence of the traveling wave for $n=1$ or from the wave $(1+t)^{-\frac{\mu}{2}}w[u_0](t,x)$ for the multidimensional case, it must hold $(t,x)\in K_R$ necessarily.
So, we have shown the validity of the property of finite speed of propagation of perturbations with constant speed $1$ (also in this case the result is already known in the literature, e.g. \cite[Corollary 2.3 in Chapter 1]{Sog08}).

\paragraph{Huygens' principle} 
In general, we have seen the existence of a forward wave front in the case of compactly supported initial data and of source term supported in the conical domain correspondingly. However, in the case of a homogeneous problem ($f\equiv 0$) a backward wave front is not present generally, even in the odd dimensional case. If we denote
\begin{align*}
u^{\Huy}(t,x) & \doteq \begin{cases} \displaystyle{2^{-1}(1+t)^{-\frac{\mu}{2}}\big(u_0(x+t)+u_0(x-t)\big) } & \mbox{if} \ \ n=1, \\
 \displaystyle{(1+t)^{-\frac{\mu}{2}} w[u_0](t,x)} & \mbox{if} \ \ n\geq 2,
 \end{cases} \\
 u^{\nHuy}(t,x) & \doteq \begin{cases} \displaystyle{2^{-\sqrt{\delta}}\int_{x-t}^{x+t} u_0(y) K_0(t,x;y;\mu,\nu^2)\, \mathrm{d}y +2^{-\sqrt{\delta}}\int_{x-t}^{x+t}\big(u_1(y)+\mu \, u_0(y)\big) K_1(t,x;y;\mu,\nu^2)\, \mathrm{d}y } & \mbox{if} \ \ n=1, \\
\displaystyle{  2^{1-\sqrt{\delta}}\int_0^t w[u_0](s,x) K_0(t,0;s; \mu,\nu^2) \, \mathrm{d}s +2^{1-\sqrt{\delta}}\int_0^t w[u_1+\mu\,  u_0](s,x) K_1(t,0;s; \mu,\nu^2) \, \mathrm{d}s} & \mbox{if} \ \ n\geq 2,
 \end{cases} 
\end{align*} then, in the term $u^{\Huy}$ we have the existence of a backward wave front set in the odd dimensional case, that is, $$\supp u^{\Huy}\subset \{(t,x)\in[0,\infty)\times \mathbb{R}^n: t-R\leq |x|\leq t+R\},$$ while in general for $u^{\nHuy}$ this is not true. We said in general, as in some special cases the kernel functions $K_0$ and $K_1$ may have simplified expressions. For example, when $\mu,\nu^2$ satisfy the condition $\delta=1$, then, the expression of the kernel $E$ is simpler than the general case, namely,
$$E(t,x;b,y;\mu,\nu^2) = (1+t)^{-\frac{\mu}{2}}(1+b)^{\frac{\mu}{2}}.$$ Therefore, for $\delta=1$ and $n\geq 3$, $n$ odd we get
\begin{align*}
u^{\nHuy}(t,x) = \frac{1}{(n-2)!!} (1+t)^{-\frac{\mu}{2}} \bigg(\frac{1}{t}\frac{\partial}{\partial t}\bigg)^{\frac{n-3}{2}} \left(t^{n-2} \fint_{\partial B_t(x)} \left( u_1(z)+\tfrac{\mu}{2}u_0(z)\right) \, \mathrm{d}\sigma_z\right),
\end{align*} where we applied simply the fundamental theorem of calculus due to the facts that $w[\varphi](s,x)$ is the $s$-derivative of a certain function involving spherical means in \eqref{wave sol operator n odd} and $K_0(t,0;s;\mu,\nu^2)=-\frac{\mu}{2}(1+t)^{-\frac{\mu}{2}}$,  $K_1(t,0;s;\mu,\nu^2)=(1+t)^{-\frac{\mu}{2}}$ do not really depend on $s$. Also, when $\delta=1$ and $n\geq 3$ is odd, the term $u^{\nHuy}$ provides a backward wave front as well and, hence, Huygens' principle holds. Curiously, in the one dimensional case even in the very special case $\delta=1$ not only Huygens' principle but also the so-called \emph{incomplete Huygens' principle} fails.
The incomplete Huygens' principle, that was introduced in \cite{Yag13}, means the presence of a backward wave front for the homogeneous equation when the second data $u_1$ is identically 0. This is due to the presence of the integral terms in \eqref{representation formula 1d case} which do not cancel each others for $\delta=1$ even though $u_1=0$ and $f=0$. 


\paragraph{Connections with other hyperbolic models}
We point out now that the range for the parameters of Gauss' hypergeometric functions in \eqref{def E(t,x;b,y)} is somehow related to the range of the corresponding parameters for the representation formula of the solution to the Cauchy problem for the Klein-Gordon equation in the anti-de Sitter space-time with complex mass, namely,
\begin{align}\label{CP de Sitter}
\begin{cases}
w_{tt}-\mathrm{e}^{2t}\Delta w + M^2 w= g(t,x), & x\in \mathbb{R}^n, t>0, \\
w(0,x)= w_0(x), &  x\in \mathbb{R}^n, \\
w_t(0,x)= w_1(x), &  x\in \mathbb{R}^n, 
\end{cases}
\end{align} where $M\in \mathbb{C}$.
In fact, considering the change of variables $$1+t\doteq \mathrm{e}^\tau, \qquad \tau =\log (1+t)$$ and the transformation $$u(t,x)\doteq \mathrm{e}^{-\frac{\mu-1}{2}\tau}v(\tau,x),$$ we have that $u$ solves \eqref{inhomog CP} if and only if $v$ solves
\begin{align*}
\begin{cases} v_{\tau \tau}-\mathrm{e}^{2\tau}\Delta v -\frac{\delta}{4}v=\mathrm{e}^{\frac{\mu+3}{2}\tau}f(\mathrm{e}^\tau-1,x), &  x\in \mathbb{R}^n, \ \tau>0,\\
v(0,x)=u_0(x), & x\in \mathbb{R}^n, \\ v_\tau(0,x)=\frac{\mu-1}{2}u_0(x) +u_1(x), & x\in \mathbb{R}^n.
\end{cases}
\end{align*}

 In particular, the case $\delta=0$ corresponds to the massless case $M=0$ in \eqref{CP de Sitter}. So, it is not surprising to find $(\frac{1}{2},\frac{1}{2};1)$ as parameters in \eqref{def E(t,x;b,y)}, having in mind the corresponding representation formula for the solution of the wave equation in the anti-de Sitter space-time (cf. \cite[equations (1.2) and (1.6)]{YagGal08}).
 

On the one hand, for $\delta>0$ the Cauchy problem \eqref{inhomog CP} can be transformed in a Cauchy problem as in \eqref{CP de Sitter} with  an imaginary mass. Therefore, we find that $(\frac{1-\sqrt{\delta}}{2},\frac{1-\sqrt{\delta}}{2};1)$ are real parameter as in the corresponding representation for the solution of \eqref{CP de Sitter} (cf.  \cite[page 682]{Yag09}). According to \cite{PalThesis}, the case $\delta>0$ corresponds to the \emph{dominant damping case}. Thus, we have that the dominant damping case for the scale-invariant wave equation is related to the Klein-Gordon equation in the anti-de Sitter space-time with imaginary mass. On the other hand, the case $\delta<0$ (classified as \emph{Klein-Gordon type case} in \cite{PalThesis} for the scale-invariant model) is related in the same way to the Klein-Gordon equation in the anti-de Sitter space-time but now with positive mass. Hence, it is not surprising that in both cases we find an analogous situation for the parameters of the hypergeometric function: indeed, there exists a complex number with nontrivial imaginary part that appears in the hypergeometric function in the first two parameters. More precisely, these complex numbers are $\frac{1-i\sqrt{-\delta}}{2}$ for \eqref{inhomog CP} and $\frac{1}{2}+iM$ for \eqref{CP de Sitter}, cf. \cite[equation (0.20)]{YagGal09}.

However, Klein-Gordon equation in the anti-de Sitter space-time (or de Sitter space-time if we consider the backwards Cauchy problem) is not the only equation which is related to \eqref{inhomog CP}. Besides the previous case, we may consider a different change of variables and transformation of the dependent variable in the case $\delta\in (0,1]$, namely,
\begin{align*}
1+t\doteq (1+\tau)^{\ell+1}, \ \  x\doteq (\ell+1)y \ \ \mbox{and} \ \ v(\tau,y)\doteq (1+t)^{\frac{\mu-1+\sqrt{\delta}}{2}}u(t,x),
\end{align*} where $\ell\doteq \frac{1-\sqrt{\delta}}{\sqrt{\delta}}$. Then, $u$ solves \eqref{inhomog CP} if and only if $v$ solves
\begin{align*}
\begin{cases} v_{\tau \tau}-(1+\tau)^{2\ell}\Delta_y v=\frac{1}{\delta }(1+\tau)^{\frac{\mu-1+\sqrt{\delta}}{2}(\ell+1)+2\ell}f((1+\tau)^{\ell+1}-1,(\ell+1)y), &  y\in \mathbb{R}^n, \ \tau>0,\\
v(0,y)=u_0(y), & y\in \mathbb{R}^n, \\ v_\tau(0,y)=\frac{\mu-1+\sqrt{\delta}}{2}(\ell+1)u_0((\ell+1)y) +(\ell+1)u_1((\ell+1)y), &y\in \mathbb{R}^n.
\end{cases}
\end{align*} Employing the representation formula given in \cite{PalRei18} for the solution of the Cauchy problem 
\begin{align}\label{CP Tricomi}
\begin{cases}
w_{tt}-(1+t)^{2\ell}\Delta w = g(t,x), & x\in \mathbb{R}^n, t>0, \\
w(0,x)= w_0(x), &  x\in \mathbb{R}^n, \\
w_t(0,x)= w_1(x), &  x\in \mathbb{R}^n, 
\end{cases}
\end{align} it is possible to find the representation formula for \eqref{inhomog CP} in the one-dimensional case. In turn, the representation formula for \eqref{CP Tricomi} in the case $n=1$ is obtained in \cite[Section 4]{PalRei18} by following the works \cite{Yag04,Yag15} on the generalized Tricomi equation (Gellerstedt equation). For a summary overview on \emph{Yagdjian's Integral Transform approach} applied to several hyperbolic equations with variable coefficients, one can see also \cite{Yag10}.

\paragraph{Future applications of the representation formulae}

In the forthcoming paper \cite{Pal19scalder}, the representation formulae which are derived in this work will be applied to study the blow-up dynamic of the semilinear wave equation with damping and mass terms in the scale-invariant case and with nonlinearity of derivative type $|\partial_t u|^p$. 

\section*{Acknowledgments}

This work is supported by the University of Pisa, Project PRA 2018 49. The  author is member of the Gruppo Nazionale per L'Analisi Matematica, la Probabilit\`{a} e le loro Applicazioni (GNAMPA) of the Instituto Nazionale di Alta Matematica (INdAM). The results in this paper have been inspired by the series of seminars \emph{Integral Transform Approach to Wave and Klein-Gordon Equations in the de Sitter space-time} held by Karen Yagdjian (UTRGV, Edinburg, Texas) during the trimester June-August 2016 at the Institute of Applied Analysis of TU Bergakademie Freiberg. Moreover, the author thanks Karen Yagdjian for his useful comments and suggestions in preparing the final version.

\end{document}